\documentclass[10pt]{article}
\usepackage{pb-diagram}

\usepackage{times}
\usepackage{amsthm}
\usepackage{amsmath}
\usepackage{amssymb}
\usepackage{mathrsfs}
\usepackage{comment}

\newcommand\ie{{\em i.e.}~}

\def\A{{\mathscr A}}
\def\a{\mathrm a}
\def\b{\mathrm b}
\def\c{\mathrm c}
\def\s{\mathrm s}
\def\t{\mathrm t}
\def\ss{\sigma}
\def\tt{\tau}
\def\ha{{\tilde \a}}
\def\hb{{\tilde \b}}
\def\hc{{\widehat \c}}

\def\haa{{\tilde \aa}}
\def\hbb{{\tilde \bb}}
\def\hcc{{\widehat \cc}}

\def\T{\mathbb{T}}

\def\la{\overrightarrow{\a}}
\def\ra{\overleftarrow{\a}}
\def\laa{\overrightarrow{\aa}}
\def\raa{\overleftarrow{\aa}}

\def\o{\omega}

\def\G{{\sf G}}
\def\hG{\tilde{\sf G}}

\def\C{\mathscr{C}}

\def\B{{\mathscr B}}

\def\H{\mathcal H}

\def\M{{\mathcal M}}

\def\S{\mathcal S}

\def\U{\mathcal U}
\def\BB{\mathfrak B}
\def\CC{{\mathfrak C}}

\def\s{{\mathsf s}}

\def\fa{\omega}
\def\fb{\nu}
\def\fc{\theta}
\def\hfa{\widehat{\omega}}
\def\hfb{\widehat{\nu}}
\def\hfc{\widehat{\theta}}

\def\Bo{\mathscr{B}}

\def\({\left(}
\def\[{\left[}
\def\){\right)}
\def\]{\right]}

\def\si{\sigma}
\def\Si{\Sigma}

\def\Aut{\mathsf{Aut}}
\def\ad{\mathsf{ad}}
\def\1{\mathfrak{1}}
\def\0{\mathfrak{0}}

\def\p{\parallel}
\def\k{\kappa}
\def\<{\langle}
\def\>{\rangle}
\def\aa{\alpha}

\def\bb{\beta}
\def\cc{\gamma}

\providecommand{\CC}{\mathfrak{C}}

\newtheorem{Theorem}{Theorem}[section]
\newtheorem{Remark}[Theorem]{Remark}
\newtheorem{Lemma}[Theorem]{Lemma}

\newtheorem{Corollary}[Theorem]{Corollary}
\newtheorem{Proposition}[Theorem]{Proposition}
\newtheorem{Definition}[Theorem]{Definition}
\newtheorem{Example}[Theorem]{Example}

\numberwithin{equation}{section}

\begin{document}


\title{$C^*$-Algebraic Covariant Structures}

\date{\today}

\author{H. Bustos  and M. M\u antoiu 
}
\date{\small}
\maketitle \vspace{-1cm}


\begin{abstract}
We introduce {\it covariant structures} $\left\{(\A,\k),(\a,\aa),\(\ha,\haa\)\right\}$ formed of a separable $C^*$-algebra $\A$, a measurable twisted action $(\a,\aa)$ of the second-countable locally compact group $\G$\,, a measurable twisted action $(\ha,\haa)$ of another second-countable locally compact group $\hG$ and a strictly continuous function $\k:\G\times\hG\to\U\M(\A)$ suitably connected with $(\a,\aa)$ and $\(\ha,\haa\)$\,. Natural notions of covariant morphisms and representations are considered, leading to a sort of twisted crossed product construction. Various $C^*$-algebras emerge by a procedure that can be iterated indefinitely and that also yields new pairs of twisted actions. Some of these $C^*$-algebras are shown to be isomorphic. The constructions are non-commutative, but are motivated by Abelian Takai duality that they eventually generalize.
\end{abstract}

\section*{Introduction}\label{duci}

Let $\A$ be a separable $C^*$-algebra with automorphism group ${\sf Aut}(\A)$, multiplier algebra $\M(\A)$ and unitary group $\U\M(\A)$ and let $\G,\hG$ be two second contable locally compact groups, with units ${\sf e}$ and $\varepsilon$ and left Haar measures $dx$ and $d\xi$ respectively. Let also $(\a,\aa)$ be a measurable twisted action of $\G$ on $\A$ and $(\ha,\haa)$ a measurable twisted action of $\hG$ on $\A$. Motivated by duality issues, we are going to investigate this pair of twisted actions in the presence of a "coupling function" $\k:\G\times\hG\to\U\M(\A)$, supposed strictly continuos. 

The simple motivating example is given by the setting involved in the well-known (Abelian) Takai duality result \cite{Ta,Tak,Ra1,Wi}. In this case $\G$ is supposed to be commutative, $\G\equiv\widehat\G$ is its Pontryagin dual and $\k(x,\xi):=\xi(x)$ is obtained by applying the character $\xi$ to the element $x$\,.  The theory starts with a single action $\a$ of the group $\G$ (let us assume it untwisted), used to construct \cite{E,EKQR,Pe,Wi} the crossed product $\B:=\A\!\rtimes\!_\a\G$\,. On this new $C^*$-algebra there is a canonical action $\widehat\b^0$ of the dual group given on elements $f$ of the dense $^*$-subalgebra $L^1(\G;\A)$ by
$$
\big[\widehat\b^0_\xi(f)\big](x):=f(x)\overline{\xi(x)}=f(x)\overline{\k(x,\xi)}\,,\ \quad\forall\,x\in\G\,,\,\xi\in\widehat\G\,.
$$
Takai's duality result states that the second generation crossed product $(\A\!\rtimes\!_\a\G)\!\rtimes_{\widehat\b^0}\!\widehat\G$ is isomorphic to the tensor product $\A\otimes\mathbb K[L^2(\G)]$ between the initial $C^*$-algebra $\A$ and the $C^*$-algebra of compact operators on the Hilbert space $L^2(\G)$\,; this isomorphism is equivariant with respect to the canonial bi-dual action on $(\A\!\rtimes\!_\a\G)\!\rtimes_{\widehat\b^0}\!\widehat\G$ and a natural product action on $\A\otimes\mathbb K[L^2(\G)]$\,.

On the other hand, this dual action is not enough if one wants to fully connect the $C^*$-algebra $\B$ with the initial $C^*$-dynamical system $(\A,\a,\G)$\,. There is also a natural strictly continuous group morphism $\lambda:\G\to\U\M(\B)$ (basically $\lambda_x=\delta_x\otimes 1$ in a suitable picture of the multiplier algebra of $\B$) and the covariance relation 
$$
\widehat\b^0_\xi(\lambda_x)=\k(x,\xi)\lambda_x
$$
holds for each $x\in\G$ and $\xi\in\widehat\G$\,. The couple $(\widehat\b^0,\lambda)$ plays an important role \cite{La,Pe} in Landstad's characterizations of the $C^*$-algebras that are isomorphic to a crossed product with group $\G$\,. But $\lambda$ can also be seen as defining an action
$$
\b:={\sf ad}_\lambda:\G\to{\sf Aut}(\B)\,,\quad\b_x(f)={\sf ad}_{\lambda_x}\!(f)\equiv\lambda_x\diamond f\diamond\lambda_{x}^\diamond\,,
$$
where $\diamond$ denotes the composition law and $^\diamond$ the involution in the (multiplier algebra of the) crossed product. Finally $\B$ comes equipped with the two actions $\b$ of the group $\G$ and $\widehat\b^0$ of the group $\widehat\G$\,. If the initial action $\a$ is twisted by a $2$-cocycle $\aa$\,, then $\lambda$ will no longer be a group morphism and $\b$ will also aquire a $2$-cocycle
$$
\bb:\G\times\G\to\U\M(\B)\,,\quad\bb(x,y):=\lambda_x\diamond\lambda_y\diamond\lambda_{xy}^\diamond\,.
$$
In addition, if initially there is also a twisted action $(\widehat\a,\widehat\aa)$ of the dual group $\widehat\G$ on $\A$, this can be converted in a modification of $\widehat\b^0$ into
$$
[\widehat\b_\xi(f)](x):=\widehat\a_\xi[f(x)]\overline{\kappa(x,\xi)}
$$
and this formula also requires a $2$-cocycle $\widehat\bb(\cdot,\cdot):=1\otimes\widehat\aa(\cdot,\cdot)$ on $\widehat\G$\,. 

The conclusion is that, for the Pontryagin couple $(\G,\widehat\G)$\,, a pair of twisted actions $\big((\a,\aa,\G),(\widehat\a,\widehat\aa,\widehat\G))$ on $\A$ generates
a pair of twisted actions $\big((\b,\bb,\G),(\widehat\b,\widehat\bb,\widehat\G))$ on the twisted crossed product \cite{BS,PR1,PR2} $\B:=\A\!\rtimes_\a^\aa\!\G$\,. A different but similar pair of twisted actions $\big((\c,\cc,\G),(\widehat\c,\widehat\cc,\widehat\G))$ arises in the same way on the other twisted crossed product $\C:=\A\!\rtimes_{\widehat\a}^{\widehat\aa}\!\widehat\G$\,. Thus two new $C^*$-algebras are available: $(\A\!\rtimes_\a^\aa\!\G)\!\rtimes_{\widehat\b}^{\widehat\bb}\!\widehat\G$ and $(\A\!\rtimes_{\widehat\a}^{\widehat\aa}\!\widehat\G)\!\rtimes_\c^\cc\!\G$\,. A very particular case of results of our section \ref{cortandui} says that they are isomorphic in a canonical very explicit way, and this implies easily an extension of Takai's result that is recovered for $\widehat\a={\sf id}$\,, $\aa=1$ and $\widehat\aa=1$\,.

The article is dedicated to extend this picture in a non-commutative setting (but since coactions are not involved, we do not obtain non-commutative versions of Takai duality \cite{IT,Qu,Ra2}). Hopefully we are going to develop and apply this elsewhere.

The first section recalls some basic facts about twisted crossed products and their unitary multipliers. 

In the second section we introduce {\it covariant structures} $\left\{(\A,\k),(\a,\aa),\(\ha,\haa\)\right\}$ formed of a separable $C^*$-algebra $\A$, a measurable twisted action $(\a,\aa)$ of the second-countable locally compact group $\G$\,, a measurable twisted action $(\ha,\haa)$ of the second-countable locally compact group $\hG$ and a strictly continuous function $\k:\G\times\hG\to\U\M(\A)$\,. We insist on the fact that $\A,\G,\hG$ can be non-commutative and the two groups $\G$ and $\hG$ are very weakly connected. At the begining we worked under rather strong assumptions: $\k$ was supposed to be a bi-character, the two "actions" $\a$ and $\ha$ were supposed to commute and each cocycle was taken to have values in the fixed-point algebra associated to the other action. Then we succeeded to isolate a much more general compatibility assumption connecting the five objects $\k,\a,\aa,\ha,\haa$\,, that is quite meaningful and allows all the subsequent developments.

In section \ref{vizigot}, this compatibility assumption is used to associate to the covariant structure $\left\{(\A,\k),(\a,\aa),\(\ha,\haa\)\right\}$ two (exterior equivalent) twisted actions $(\la,\laa)$ and $(\ra,\raa)$ of the product group $\G\times\hG$ on $\A$\,.

In section \ref{vizigott} we define the (twisted crossed) {\it bi-product} of a covariant structure $\left\{(\A,\k),(\a,\aa),\(\ha,\haa\)\right\}$ by an universal property involving {\it covariant morphisms}; these are triples $(r,u,v)$ such that $(r,u)$ is a covariant morphism of the twisted $C^*$-dynamical system $(\A,\a,\aa,\G)$\,, $(r,v)$ is a covariant morphism of the twisted $C^*$-dynamical system $(\A,\ha,\haa,\hG)$ and the commutation between $u_x$ and $v_\xi$ is ruled by the coupling function $\k$\,. Since such covariant morphisms are rigidly related to usual covariant morphisms of the twisted action $(\la,\laa)$\,, existence of bi-products follows easily from the theory of twisted crossed products; one can see $\A\!\rtimes_{\la}^{\laa}\!(\G\times\hG)$ as one of its possible realizations.

The remaining part of the paper is dedicated to other realizations, involving iterated twisted crossed products; this will make the connection with the first half of the Introduction.

In section \ref{vipsgrot}, associated to a covariant structure $\left\{(\A,\k),(\a,\aa),\(\ha,\haa\)\right\}$\,, we introduce the first generation covariant structures $\left\{(\A\!\rtimes_\a^\aa\!\G,{\sf k}),(\b,\bb),\(\hb,\hbb\)\right\}$ and $\left\{(\A\!\rtimes_{\tilde\a}^{\tilde\aa}\!\tilde\G,\tilde{\sf k}),(\c,\cc),\(\tilde\c,\tilde\cc\)\right\}$ and then the second generation twisted crossed products $(\A\!\rtimes_\a^\aa\!\G)\!\rtimes_{\tilde\b}^{\tilde\bb}\!\tilde\G$ and $(\A\!\rtimes_{\tilde\a}^{\tilde\aa}\!\tilde\G)\!\rtimes_\c^\cc\!\G$\,. Checking the axioms relies heavily on the compatibility assumption between $\k,\a,\aa,\ha,\haa$\,.

The main result is contained in section \ref{cortandui}. It is shown that the following isomorphisms hold
\begin{equation}\label{mein}
\A\!\rtimes_{\la}^{\laa}\!(\G\times\hG)\cong\A\!\rtimes_{\ra}^{\raa}\!(\G\times\hG)\cong(\A\!\rtimes_\a^\aa\!\G)\!\rtimes_{\tilde\b}^{\tilde\bb}\!\tilde\G\cong(\A\!\rtimes_{\tilde\a}^{\tilde\aa}\!\tilde\G)\!\rtimes_\c^\cc\!\G\,.
\end{equation}
This is obtained both by studying the covariant representations of all the structures involved and (for explicitness) by comparing the concrete form of the composition laws. All the four algebras above can be regarded as realizations of the bi-product attached to the covariant structure $\left\{(\A,\k),(\a,\aa),\(\ha,\haa\)\right\}$\,. The isomorphisms in (\ref{mein}) even hold in the category of covariant structures.

Some examples are presented in section \ref{corsandui}. In particular, it is shown how a twisted version of the Abelian duality result can be deduced from the last isomorphism in (\ref{mein}).

\newpage

\section{Twisted actions}\label{fourtundoy}

\medskip
\begin{Definition}\label{carcharoth}
{\rm A twisted action} of the locally compact group $\,\G$ on the $C^*$-algebra $\A$ is a pair $(\mathrm a,\alpha)$ composed of mappings $\mathrm a:\G\rightarrow\Aut(\A)$ and $\alpha:\G\times\G\rightarrow\mathcal U\mathcal M(\A)$ such that
\begin{equation*}\label{elwe}
\mathrm a_{\sf e}={\rm id}_\A\,,\quad\ \mathrm a_x\circ\mathrm a_y={\sf ad}_{\alpha(x,y)}\circ\mathrm a_{xy}\,,\ \quad\forall\,x,y\in\G\,,
\end{equation*}
\begin{equation*}\label{manwe}
\alpha(x,{\sf e})=1=\alpha({\sf e},x)\,,\ \quad\forall\,x\in\G\,,
\end{equation*}
\begin{equation*}\label{orome}
\alpha(x,y)\,\alpha(xy,z)=\mathrm a_{x}[\alpha(y,z)]\,\alpha(x,yz)\,,\ \quad\forall\,x,y,z\in\G\,.
\end{equation*}
If $\mathrm a$ is strongly measurable and $\alpha$ is strictly measurable we speak of {\rm a measurable twisted action}. 
If $\mathrm a$ is strongly continuous and $\alpha$ is strictly continuous we speak of {\rm a continuous twisted action}. 
\end{Definition}

\medskip
To a measurable twisted action $(\a,\aa)$ of the group $\sf G$ on the $C^*$-algebra $\A$\ one associates \cite{BS,PR1} the Banach $^*$-algebra $L^1_{\a,\aa}(\G;\A)\equiv L^1(\G;\A)$ (cf. \cite[App. B]{Wi}) and its enveloping $C^*$-algebra, the twisted crossed  product $\A\!\rtimes_\a^\aa\!\G$\,. The norm on $L^1(\G;\A)$ is $\p\!f\!\p_{1}\,:=\!\int_\G\!dx\!\p\!f(x)\!\p_\A$\,. The composition laws are
\begin{equation*}\label{ascar}
\begin{aligned}
\(f\diamond g\)\!(x):=\int_\G\!dy\,f(y)\,\mathrm a_y\!\[g(y^{-1}x)\]\alpha(y,y^{-1}x)\,,
\end{aligned}
\end{equation*}
\begin{equation*}\label{amdram}
f^{\diamond}(x):=\Delta_\G(x)^{-1}\aa(x,x^{-1})^*\a_x[f(x^{-1})^*]\,.
\end{equation*}

We recall that the non-degenerate representations of $\A\!\rtimes_\a^\aa\!\G$ are in one-to one correspondence with covariant representations of the twisted $C^*$-dynamical system $(\A,\a,\aa)$\,. These are triples $(\H,\pi,U)$ where $\H$ is a Hilbert space, $\pi:\A\rightarrow\mathbb B(\H)$ a non-degenerate representation of $\A$ by bounded operators in $\H$ and $U:\G\to\mathbb U(\H)$ a strongly measurable map whose values are unitary operators in $\H$\,, satisfying
\begin{equation*}\label{ainulindale}
U_xU_y=\pi[\aa(x,y)]U_{xy}\,,\quad\ \forall\,x,y\in\G\,,
\end{equation*}
\begin{equation*}\label{erendil}
U_x\pi(A)U_x^*=\pi[\a_x(A)]\,,\quad\ \forall\,x\in\G\,,\,A\in\A\,.
\end{equation*}
The representation $\pi\!\rtimes\!U$ corresponding to $(\H,\pi,U)$ (its integrated form) acts on $f\in L^1(\G;\A)$ as
\begin{equation*}\label{balegar}
(\pi\!\rtimes\!U)f:=\int_\G\!dx\,\pi[f(x)]\,U_x\,.
\end{equation*}

We also recall that {\it a covariant morphism} of $(\A,\a,\aa)$ \cite[Sect. 1]{PR2} is composed of a $C^*$-algebra $\B$\,, a non-degenerate morphism $r:\A\to\M(\B)$ and a strictly measurable map $u:\G\to\U\M(\B)$ satisfying for $x,y\in\G$ and $A\in\A$ the relations 
\begin{equation*}\label{flamanzila}
u_x r(A)u_x^*=r[\a_x(A)]\,,\quad\ u_x u_y=r[\aa(x,y)]u_{xy}\,.
\end{equation*}

\begin{Remark}\label{strigoy}
{\rm Defining the twisted crossed product as the enveloping $C^*$-algebra of the $L^1$ Banach algebra will be convenient in the setting of our article. Occasionally we are going to use the fact that this enveloping algebra has universal properties (cf. \cite[Sect. 2]{PR1} and \cite[Sect. 1]{PR2}), which can be used as alternative definitions. 
}
\end{Remark}

\medskip
Some considerations about unitary multipliers of twisted crossed products will be needed.
It is true \cite[Prop. 4.19]{Bu} that all the unitary multipliers of $L^1_{\a,\aa}(\G;\A)$ have the form $\delta_z\otimes m$\,, where $\delta_z$ is the Dirac measure in $z\in\G$ and $m\in\mathcal U\mathcal M(\A)$\,. One can find in \cite{Bu} many other results about the interpretation of multiplier-valued regular measures on $\G$ with bounded variation as (left or bi-sided) multipliers on $L^1_{\a,\aa}(\G;\A)$\,. Since we only need simple facts, and since the connection between the multipliers of a Banach $^*$-algebra and the multipliers of its enveloping $C^*$-algebra can be murky even in simple situations \cite{Lai}, we are going to give an independent treatment.

If $z\in\G$ and $m$ is a multiplier of $\A$\, the meaning of $\delta_z\otimes m$ as a measure with values in $\mathcal M(\A)$ is obvious. To it we associate the operators $(\delta_z\otimes m)_l,(\delta_z\otimes m)_r:L^1(\G;\A)\rightarrow L^1(\G;\A)$ given by
\begin{equation}\label{ossiriand}
\[(\delta_z\otimes m)_l g\]\!(x)\equiv\[(\delta_z\otimes m)\diamond g\]\!(x):=m\,\a_z\!\[g(z^{-1}x)\]\aa(z,z^{-1}x)\,,
\end{equation}
\begin{equation}\label{dorthonion}
\[(\delta_z\otimes m)_r f\]\!(x)\equiv\[f\diamond(\delta_z\otimes m)\]\!(x):=f(xz^{-1})\,\a_{xz^{-1}}(m)\,\aa(xz^{-1},z)\,.
\end{equation}
One checks easily that $\left\{(\delta_z\otimes m)_l,(\delta_z\otimes m)_r\right\}$ is a double centralizer of the Banach $^*$-algebra $L^1_{\a,\aa}(\G;\A)$\,, i.e. 
\begin{equation}\label{athrabeth}
f\diamond\[(\delta_z\otimes m)_l g\]=\[(\delta_z\otimes m)_r f\]\diamond g\,,\quad\ \forall\,f,g\in L^1(\G;\A)\,.
\end{equation}
The particular case $z={\sf e}$ is worth mentioning:
\begin{equation}\label{himlad}
\[(\delta_{\sf e}\otimes m)\diamond f\diamond(\delta_{\sf e}\otimes n)\]\!(x)=mf(x)\,\a_x(n)\,.
\end{equation}

From now on we assume that $m$ is a unitary multiplier of $\A$\,. To show that $\delta_z\otimes m$ extends to a multiplier of the full twisted crossed product, one has to examine its behavior under the integrated form $\Pi:=\pi\!\rtimes\!U$ of an arbitrary covariant representations $(\pi,U,\H)$\,. One has
$$
\begin{aligned}
\Pi\[(\delta_z\otimes m)_l g\]&=\int_\G \!dx\,\pi\!\left\{m\,\a_z[g(z^{-1}x)]\aa(z,z^{-1}x)\right\}U_x\\
&=\pi(m)\,U_z\!\int_\G \!dx\,\pi[g(z^{-1}x)]\,U_z^*\,\pi\!\[\aa(z,z^{-1}x)\]U_x\\
&=\pi(m)\,U_z\!\int_\G \!dy\,\pi[g(y)]\,U_z^*\,\pi\!\[\aa(z,y)\]U_{zy}\\
&=\pi(m)\,U_z\!\int_\G \!dy\,\pi[g(y)]\,U_y=\pi(m)\,U_z\,\pi(g)\,.
\end{aligned}
$$
Then, since $U_z$ and $\pi(m)$ are unitary operators, one gets
$$
\p\!\Pi\[(\delta_z\otimes m)_l g\]\!\p_{\mathbb B(\H)}\,=\,\p\!\Pi(g)\!\p_{\mathbb B(\H)}
$$
so $(\delta_z\otimes m)_l$ extends to an isometry of the enveloping $C^*$-algebra $\A\!\rtimes_\a^\aa\G$\,.
A similar statement holds for $(\delta_z\otimes m)_r$\,, based on the identity $\Pi\[(\delta_z\otimes m)_r f\]=\Pi(f)\,\pi(m)U_z$\,.
Then, by continuity and density, the two extensions form a double centralizer of $\A\!\rtimes_\a^\aa\!\G$\,.

A shorter way to express the two computations above is to write $(\pi\!\rtimes\!U)(\delta_z\otimes m)=\pi(m)U_z$\,.
One can deduce from this (or from many other arguments) the algebra of these unitary multipliers:
\begin{equation}\label{mistymountains}
(\delta_y\otimes n)\diamond(\delta_z\otimes m)=\delta_{yz}\otimes \[n\a_y(m)\aa(y,z)\],
\end{equation}
\begin{equation}\label{huan}
(\delta_z\otimes m)^\diamond=\delta_{z^{-1}}\otimes \[\aa(z^{-1},z)^*\a_{z^{-1}}(m^*)\].
\end{equation}
Later on we are going to need the particular case
\begin{equation}\label{ladros}
(\delta_{\sf e}\otimes m)^{\diamond}=\delta_{\sf e}\otimes m^*.
\end{equation}

We close this section with two remarks that will be useful later.

\medskip
\begin{Remark}\label{ghiveci}
{\rm Let $\G,\hG$ be two locally compact groups and $(c,\gamma)$ a twisted action of $\G\times\hG$ on the $C^*$-algebra $\A$\,. Define $c^\dagger$ and $\gamma^\dagger$ respectively by $c^\dagger_{(\xi,x)}:=c_{(x,\xi)}$ and $\gamma^\dagger\big((\xi,x),(\eta,y)\big):=\gamma\big((x,\xi),(y,\eta)\big)$\,. Then $(c^\dagger,\gamma^\dagger)$ is a twisted action of the group $\hG\times\G$ on $\A$\,. The twisted crossed products $\A\!\rtimes_\c^\gamma\!(\G\times\hG)$ and $\A\!\rtimes_{\c^\dagger}^{\gamma^\dagger}\!(\hG\times\G)$ are isomorphic and at the level of $L^1$-elements the isomorphism is just composing with the flip $(x,\xi)\to(\xi,x)$\,.
}
\end{Remark}

\medskip
\begin{Remark}\label{duniath}
{\rm We say that the two twisted actions $(\b,\beta)$ and $(\b',\beta')$ are {\it exterior equivalent} \cite{PR1} if there exists a strictly measurable map (a normalized $1$-cochain) $q:{\sf G}\rightarrow\mathcal U\mathcal M(\A)$ such that $q({\sf e})=1$ and
\begin{equation*}\label{melian}
\b'_x={\sf ad}_{q_x}\circ\b_x\,,\ \quad\forall\,x\in{\sf G}\,,
\end{equation*}
\begin{equation*}\label{thingol}
\beta'(x,y)=q_x\b_x(q_y)\beta(x,y)q_{xy}^*\,,\ \quad\forall\,x,y\in{\sf G}\,.
\end{equation*}
In such a situation we are going to write $(\b,\beta)\overset{q}{\sim}(\b',\beta')$\,. It is easy to see that $\sim$ is an equivalence relation.

Let us suppose that $(\b,\beta)\overset{q}{\sim}(\b',\beta')$\,. Then \cite[Lemma 3.3]{PR1} the twisted crossed products $\A\!\rtimes_\b^\beta\!{\sf G}$ and $\A\!\rtimes_{\b'}^{\beta'}\!{\sf G}$ are canonically isomorphic. At the level of $L^1({\sf G};\A)$ the isomorphism acts as $[\iota_q(f)](x):=f(x)q_x^*$\,.
 }
\end{Remark}

\section{Covariant structures}\label{fourdoy}

Two second countable locally compact group are given: $\G$ with elements $x,y,z$\,, unit ${\sf e}$ and Haar measure $dx$ and $\hG$ which has elements $\xi,\eta,\zeta$\,, unit $\varepsilon$ and Haar measure $d\xi$\,. The next definition is provisory; the really useful concept is that of Definition \ref{stalceala}.

\medskip
\begin{Definition}\label{cacarot}
A {\rm semi-covariant structure} $\left\{(\A,\k),(\a,\aa),\(\ha,\haa\)\right\}$ is given by a separable $C^*$-algebra $\A$ endowed with two measurable twisted action $(\a,\aa)$ of $\,\G$ and $(\ha,\haa)$ of $\,\hG$\, respectively, and with a strictly continuous map 
\begin{equation*}\label{zdrelita}
\G\times\hG\ni(x,\xi)\mapsto\k(x,\xi)\in\U\M(\A)
\end{equation*}
satisfying the normalization conditions 
\begin{equation*}\label{strivita}
\k({\sf e},\xi)=1=\k(x,\varepsilon)\,,\quad\ \forall\,x\in\G\,,\,\xi\in\hG\,.
\end{equation*}  
\end{Definition}

When extra regularity properties (as continuity) of the twisted actions will be present, this will usually be specified. One could call $\k$ {\it the coupling function}.

\medskip
\begin{Definition}\label{stirba}
We call {\rm covariant morphism of the semi-covariant structure} $\left\{(\mathscr A,\k),(\a,\aa),(\tilde\a,\tilde\aa)\right\}$ a quadruplet $(\B,r,u,v)$ where 
\begin{enumerate}
\item
$(\mathscr B,r,u)$ is a covariant morphism of the twisted $C^*$-dynamical system $(\A,\a,\aa)$ with group $\G$\,, 
\item
$(\B,r,v)$ is a covariant morphism of the twisted $C^*$-dynamical system $(\A,\ha,\haa)$ with group $\tilde\G$\,,
\item
the commutation relation $u_x v_\xi=r[\k(x,\xi)] v_\xi u_x$ holds for every $(x,\xi)\in\G\times\hG$\,.
\end{enumerate}
If $\,\B=\mathbb K(\H)\,$ for some Hilbert space $\H$ (thus $\M(\B)=\mathbb B(\H)$) we speak of {\rm a covariant representation} and we use notations as $(\H,\pi,U,V)$\,.
\end{Definition}

\medskip
Let us investigate under which assumptions convenient covariant morphisms exists. For a hypothetical one $(\B,r,u,v)$ with {\it faithful} $r$ and for $A\in\A,x\in\G\,,\xi\in\hG$ one has
$$
\begin{aligned}
(v_\xi u_x) r(A)(v_\xi u_x)^*&=v_\xi r[\a_x(A)]v_\xi^*=r\{\ha_\xi[\a_x(A)]\}
\end{aligned}
$$
but also
$$
\begin{aligned}
v_\xi u_x r(A)(v_\xi u_x)^*&=r[\k(x,\xi)^*]\,u_x v_\xi\,r(A)\,v_\xi^* u_x^*\,r[\k(x,\xi)]\\
&=r[\k(x,\xi)^*]\,u_x r[\ha_\xi(A)] u_x^*\,r[\k(x,\xi)]\\
&=r\big\{\k(x,\xi)^*\a_x[\ha_\xi(A)]\k(x,\xi)\big\}\,.
\end{aligned}
$$
it follows that for all $x,\xi$ one must have
\begin{equation}\label{dihanie}
\a_x\circ\ha_\xi=\ad_{\k(x,\xi)}\circ\ha_\xi\circ\a_x\,,
\end{equation}
so $\ad_{\k(\cdot,\cdot)}$ measures the non-commutativity of the actions. If $\k$ is center-valued the actions do commute.

Now, for arbitrary $x,y\in\G$\,, $\xi,\eta\in\hG$ let us compute $v_\xi u_x v_\eta u_y$ in two ways. First
$$
\begin{aligned}
v_\xi u_x v_\eta u_y&=v_\xi r[\k(x,\eta)]v_\eta u_x u_y \\
&=r\!\left\{\ha_\xi\!\[\k(x,\eta)\]\right\}v_\xi v_\eta u_x u_y\\
&=r\!\left\{\ha_\xi\!\[\k(x,\eta)\]\right\}r[\haa(\xi,\eta)]v_{\xi\eta}\,r[\aa(x,y)]u_{xy}\\
&=r\big\{\ha_\xi\!\[\k(x,\eta)\]\haa(\xi,\eta)\,\ha_{\xi\eta}\!\[\aa(x,y)\]\!\big\}v_{\xi\eta}u_{xy}\,.
\end{aligned}
$$
But on the other hand
$$
\begin{aligned}
v_\xi u_x v_\eta u_y&=r[\k(x,\xi)^*]u_x v_\xi\,r[\k(y,\eta)^*]\,u_y v_\eta \\
&=r[\k(x,\xi)^*]u_x r\{\ha_\xi[\k(y,\eta)^*]\}v_\xi u_y v_\eta\\
&=r[\k(x,\xi)^*] r\{(\a_x\circ\ha_\xi)[\k(y,\eta)^*]\}u_x r[\k(y,\xi)^*]u_y v_\xi v_\eta\\
&=r[\k(x,\xi)^*] r\{(\a_x\circ\ha_\xi)[\k(y,\eta)^*]\}r\{\a_x[\k(y,\xi)^*]\}u_x u_y v_\xi v_\eta\\
&=r\big\{\k(x,\xi)^*(\a_x\circ\ha_\xi)[\k(y,\eta)^*]\,\a_x[\k(y,\xi)^*]\big\}\,r[\aa(x,y)]u_{xy} r[\haa(\xi,\eta)] v_{\xi\eta}\\
&=r\big\{\k(x,\xi)^*(\a_x\circ\ha_\xi)[\k(y,\eta)^*]\,\a_x[\k(y,\xi)^*]\aa(x,y)\}\,r\{\a_{xy}[\haa(\xi,\eta)]\big\} u_{xy} v_{\xi\eta}\\
&=r\big\{\k(x,\xi)^*(\a_x\circ\ha_\xi)[\k(y,\eta)^*]\,\a_x[\k(y,\xi)^*]\aa(x,y)\,\a_{xy}[\haa(\xi,\eta)]\k(xy,\xi\eta)\big\}v_{\xi\eta}u_{xy}\,.
\end{aligned}
$$
The conclusion, valid for every $x,y,\xi,\eta$ is 
\begin{equation}\label{vrajitroaca}
\ha_\xi\!\[\k(x,\eta)\]\haa(\xi,\eta)\,\ha_{\xi\eta}\!\[\aa(x,y)\]=\k(x,\xi)^*(\a_x\circ\ha_\xi)[\k(y,\eta)^*]\,\a_x[\k(y,\xi)^*]\aa(x,y)\,\a_{xy}[\haa(\xi,\eta)]\k(xy,\xi\eta)\,.
\end{equation}
The cohomological interpretation of (\ref{vrajitroaca}) will be seen in Remark \ref{strufat}. This relation is sometimes hard to use, so we will reduce to it to a pair of simpler ones (also having a cohomological meaning). By taking $y={\sf e}$ one gets
\begin{equation}\label{poarca}
\a_x[\haa(\xi,\eta)]=\k(x,\xi)\ha_\xi\!\[\k(x,\eta)\]\haa(\xi,\eta)\k(x,\xi\eta)^*
\end{equation}
and by taking $\eta=\varepsilon$ one gets
\begin{equation}\label{spoarca}
\ha_\xi[\aa(x,y)]=\k(x,\xi)^*\a_x\!\[\k(y,\xi)^*\]\aa(x,y)\k(xy,\xi)\,.
\end{equation}

\begin{Lemma}\label{strikat}
Assume that $(\a,\aa)$ is a twisted action of $\,\G$ and $(\ha,\haa)$ is a twisted action of $\,\hG$\,, satisfying (\ref{dihanie}) for every $x,\xi$\,.
Then (\ref{vrajitroaca}) holds for every $x,y,\xi,\eta$ if and only if (\ref{poarca}) and (\ref{spoarca}) hold for every $x,y,\xi,\eta$\,.
\end{Lemma}

\begin{proof}
We only need to deduce (\ref{vrajitroaca}) from (\ref{poarca}) and (\ref{spoarca}). One transforms the r.h.s.
$$
\begin{aligned}
&\ \ \ \ \ \ \ \ \ \k(x,\xi)^*(\a_x\circ\ha_\xi)[\k(y,\eta)^*]\,\a_x[\k(y,\xi)^*]\,\aa(x,y)\,\a_{xy}[\haa(\xi,\eta)]\,\k(xy,\xi\eta)\\
&\overset{(\ref{poarca})}{=}\k(x,\xi)^*(\a_x\circ\ha_\xi)[\k(y,\eta)^*]\,\a_x[\k(y,\xi)^*]\,\aa(x,y)\,\k(xy,\xi)\,\ha_\xi[\k(xy,\eta)]\,\haa(\xi,\eta)\\
&\overset{(\ref{dihanie})}{=}(\ha_{\xi}\circ\a_x)[\k(y,\eta)^*]\,\k(x,\xi)^*\,\a_x[\k(y,\xi)^*]\,\aa(x,y)\,\k(xy,\xi)\,\ha_\xi[\k(xy,\eta)]\,\haa(\xi,\eta)\\
&\overset{(\ref{spoarca})}{=}(\ha_{\xi}\circ\a_x)[\k(y,\eta)^*]\,\ha_\xi[\aa(x,y)]\,\ha_\xi[\k(xy,\eta)]\,\haa(\xi,\eta)\\
&\ \ =\ \,\ha_{\xi}\!\left\{\a_x[\k(y,\eta)^*]\aa(x,y)\k(xy,\eta)\right\}\haa(\xi,\eta)\\
&\overset{(\ref{spoarca})}{=}\ha_{\xi}\!\left\{\k(x,\eta)\,\ha_\eta[\aa(x,y)]\right\}\haa(\xi,\eta)\\
&\ \ =\ \ha_\xi[\k(x,\eta)]\,(\ha_\xi\circ\ha_\eta)[\aa(x,y)]\,\haa(\xi,\eta)\\
&\ \ =\ \ha_\xi\!\[\k(x,\eta)\]\haa(\xi,\eta)\,\ha_{\xi\eta}\!\[\aa(x,y)\]
\end{aligned}
$$
and we are done.
\qed
\end{proof}

Now we have at least one motivation for our main notion; see also Remarks \ref{mamaliga} and \ref{boilup} and the constructions of the next sections.

\medskip
\begin{Definition}\label{stalceala}
{\rm A covariant structure} is a semi-covariant structure $\{(\A,\k),(\a,\aa),(\ha,\haa)\}$ for which relations (\ref{dihanie}), (\ref{poarca}) and (\ref{spoarca}) are satisfied for all elements $x,y\in\G\,,\,\xi,\eta\in\hG$\,.
\end{Definition}

\medskip
\begin{Example}\label{ghivreci}
{\rm Suppose that for every $x,\xi$ the multiplier $\k(x,\xi)$ is central and a fixed point for both $\a$ and $\ha$ (this happens if $\k(x,\xi)\in\T$ for instance). Also assume that it is "bilinear" (multiplicative in the second variable and anti-multiplicative in the first). Then (\ref{dihanie}), (\ref{poarca}) and (\ref{spoarca}) simplify a lot: the two actions commute and the cocycles of each twisted action are fixed points of the other action. {\it A sub-particular case is one of the motivations of all our constructions:} $\G$ is an Abelian locally compact group, $\hG:=\widehat\G$ is its Pontryagin dual and $\k(x,\xi):=\xi(x)$ is obtained by applying the character $\xi$ to the element $x$\,.
}
\end{Example}

\medskip
\begin{Example}\label{farharad}
{\rm Obviously a twisted action of $\G$ (or of $\hG$) can be completed by trivial objects to get a covariant structure. One might call $\left\{(\A,1),({\rm id},1),\(\ha,\haa\)\right\}$ {\it a $\G$-trivial covariant structure} and $\left\{(\A,1),(\a,\aa),\({\rm id},1\)\right\}$ might be called {\it a $\hG$-trivial covariant structure}. Similar examples with some non-trivial $\k$ are also available.}
\end{Example}

\medskip
\begin{Example}\label{besteleala}
{\rm We outline now an example that will play an important role below. Let $(\ha,\haa)$ be a measurable twisted action of $\hG$ on the $C^*$-algebra $\A$ and let $\rho$ be a $1$-cochain on $\G$ with values in $\mathcal U\mathcal M(\A)$\,, i.e. a map $\rho:\G\rightarrow\mathcal U\mathcal M(\A)$ satisfying $\rho_{\sf e}=1$\,. The family $\{(\A,\k),(\rho),(\tilde\a,\haa)\}$ will be called {\it a $\G$-particular covariant structure} if for $x\in\G$ and $\xi\in\hG$ one has {\it the covariance condition}
\begin{equation}\label{gargaroth}
\tilde\a_\xi(\rho_x)=\k(x,\xi)^*\rho_x\,.
\end{equation}
If $\G$ is commutative, $\hG$ is its dual, $\k(x,\xi):=\xi(x)$\,, $\,\haa=1$ (so $\ha$ is a true action) and $\rho$ is a group morphism, $(\A,\rho,\tilde\a)$ is traditionally called $\G$-{\it product}; then the condition (\ref{gargaroth}) plays an important role in Landstad duality theory \cite{La,Pe}.
}
\end{Example}

\medskip
\begin{Lemma}\label{andaram}
A $\G$-particular covariant structure can be turned into a covariant structure.
\end{Lemma}

\begin{proof}
If $\left\{\A,(\rho),(\tilde\a,\haa)\right\}$ is a particular covariant structure, let us set 
\begin{equation*}\label{jmechera}
\a_x:={\sf ad}_{\rho_x}\quad {\rm and}\quad \aa(x,y):=\rho_x\rho_y\rho_{xy}^*\,.
\end{equation*} 
Clearly $(\a,\aa)$ is a twisted action of $\G$ on $\A$\,. It is easy to check that it is measurable if $\rho$ is strictly measurable and continuous if $\rho$ is strictly continuous. 

To check (\ref{dihanie}), for $x\in\G\,,\xi\in\hG$ one computes
$$
\ha_\xi\circ\ad_{\rho_x}=\ad_{\ha_\xi(\rho_x)}\!\circ\ha_\xi=\ad_{\k(x,\xi)^*\rho_x}\!\circ\ha_\xi=\ad_{\k(x,\xi)^*}\circ\ad_{\rho_x}\!\circ\ha_\xi\,.
$$
We now verify (\ref{spoarca}):
$$
\begin{aligned}
\k(x,\xi)^*\a_{x}\!\[\k(y,\xi)^*\]\aa(x,y)\k(xy,\xi)&=\k(x,\xi)^*\rho_{x}\k(y,\xi)^*\rho_{x}^*\,\rho_{x}\rho_{y}\rho_{xy}^*\,\k(xy,\xi)\\
&=\k(x,\xi)^*\rho_{x}\,\k(y,\xi)^*\rho_{y}\,\rho_{xy}^*\,\k(xy,\xi)\\
&=\tilde\a_\xi(\rho_x)\,\tilde\a_\xi(\rho_y)\,\tilde\a_\xi(\rho_{xy})^*=\ha_\xi[\aa(x,y)]\,.
\end{aligned}
$$
The relation (\ref{poarca}) reads now
\begin{equation}\label{poarcaprim}
\rho_x\haa(\xi,\eta)\rho_x^*=\k(x,\xi)\ha_\xi\!\[\k(x,\eta)\]\haa(\xi,\eta)\k(x,\xi\eta)^*\,.
\end{equation}
Rewriting (\ref{gargaroth}) in the form $\k(x,\xi)^*=\tilde\a_\xi(\rho_x)\rho_x^*$\,, the r.h.s of (\ref{poarcaprim}) can be transformed
$$
\begin{aligned}
\k(x,\xi)\ha_\xi\!\[\k(x,\eta)\]\haa(\xi,\eta)\k(x,\xi\eta)^*&=\rho_x\tilde\a_\xi(\rho_x^*)\,\ha_\xi\!\[\rho_x\tilde\a_\eta(\rho_x^*)\]\haa(\xi,\eta)\,\tilde\a_{\xi\eta}(\rho_x)\rho_x^*\\
&=\rho_x\,\ha_\xi\!\[\tilde\a_\eta(\rho_x^*)\]\haa(\xi,\eta)\,\tilde\a_{\xi\eta}(\rho_x)\rho_x^*\\
&=\rho_x\,\haa(\xi,\eta)\,\tilde\a_{\xi\eta}(\rho_x^*)\,\tilde\a_{\xi\eta}(\rho_x)\rho_x^*\\
&=\rho_x\haa(\xi,\eta)\rho_x^*\,.
\end{aligned}
$$
\end{proof}

\begin{Example}\label{rhovanion}
{\rm By analogy, one defines {\it $\hG$-particular (measurable) covariant structures} $\left\{(\A,\k),(\a,\aa),\(\ha,\haa\)\right\}$ where, by definition, the twisted action $(\a,\aa)$ is arbitrary, but one has $\ha_\xi:={\sf ad}_{\tilde\rho_\xi}$ and $\haa(\xi,\eta):=\tilde\rho_\xi\tilde\rho_\eta\tilde\rho_{\xi\eta}^{\,*}$ for some measurable $1$-cochain $\tilde\rho:\hG\to\mathcal U\mathcal M(\A)$ satisfying $\a_x(\tilde\rho_\xi)=\k(x,\xi)\tilde\rho_\xi$ for all $x,\xi$\,.}
\end{Example}

\medskip
\begin{Example}\label{papaleasca}
{\rm We close this section giving an example of covariant representation of a given covariant structure $\{(\A,\k),(\a,\aa),(\ha,\haa)\}$\,. Let $\varpi:\A\to\mathbb B(\H)$ be a faithful representation in a separable Hilbert space $\H$\,. We can inflate $\varpi$ to a representation of $\A$ in the Hilbert space $\mathscr H:=L^2(\G\times\hG;\H)\cong L^2(\G\times\hG)\otimes\H$ by
\begin{equation}\label{fleasca}
[\pi(A)\Omega](x,\xi):=\varpi\big[(\ha_\xi\circ\a_x)(A)\big]\Omega(x,\xi)\,.
\end{equation}
One also defines
\begin{equation}\label{pleasca}
(U_z\Omega)(x,\xi):=\Delta_\G(z)^{1/2}\varpi\big\{\ha_\xi[\aa(x,z)]\big\}\Omega(xz,\xi)\,,
\end{equation}
\begin{equation}\label{pleasca}
(V_\zeta\Omega)(x,\xi):=\Delta_{\hG}(\zeta)^{1/2}\varpi\big\{\ha_\xi[\k(x,\zeta)]\haa(\xi,\zeta)\big\}\Omega(x,\xi\zeta)\,.
\end{equation}
It is quite straightforward to show that $(\mathscr H,\pi,U,V)$ is indeed a covariant representation; we say that {\it it is induced by $\varpi$}\,. Let us only indicate the most difficult of the relevant computations:
$$
\begin{aligned}
(U_z V_\zeta\Omega)(x,\xi)&=\Delta_\G(z)^{1/2}\varpi\big\{\ha_\xi[\aa(x,z)]\big\}(V_\zeta\Omega)(xz,\zeta)\\
&=\Delta_\G(z)^{1/2}\varpi\big\{\ha_\xi[\aa(x,z)]\big\}\Delta_{\hG}(z)^{1/2}\varpi\big\{\ha_\xi[\k(xz,\zeta)]\haa(\xi,\zeta)\big\}\Omega(xz,\xi\zeta)\\
&=\Delta_\G(z)^{1/2}\Delta_{\hG}(\zeta)^{1/2}\varpi\big\{\ha_\xi[\aa(x,z)\k(xz,\zeta)]\big\}\varpi[\haa(\xi,\zeta)]\Omega(xz,\xi\zeta)\\
&\!\overset{(\ref{spoarca})}{=}\!\Delta_\G(z)^{1/2}\Delta_{\hG}(\zeta)^{1/2}\varpi\big\{\ha_\xi\big[\a_x\big(\k(z,\zeta)\big)\k(x,\zeta)\,\ha_\zeta\big(\aa(x,z)\big)\big]\big\}\varpi[\haa(\xi,\zeta)]\Omega(xz,\xi\zeta)\\
&=\Delta_\G(z)^{1/2}\Delta_{\hG}(\zeta)^{1/2}\varpi\big\{(\ha_\xi\circ\a_x)[\k(z,\zeta)]\big\}\varpi\big\{\ha_\xi[\k(x,\zeta)](\ha_\xi\circ\ha_\zeta)[\aa(x,z)]\haa(\xi,\zeta)\big\}\Omega(xz,\xi\zeta)\\
&=\varpi\big\{(\ha_\xi\circ\a_x)[\k(z,\zeta)]\big\}\Delta_\G(z)^{1/2}\Delta_{\hG}(\zeta)^{1/2}\varpi\big\{\ha_\xi[\k(x,\zeta)]\haa(\xi,\zeta)\ha_{\xi\zeta}[\aa(x,z)]\big\}\Omega(xz,\xi\zeta)\\
&=\varpi\big\{(\ha_\xi\circ\a_x)[\k(z,\zeta)]\big\}\Delta_{\hG}(\zeta)^{1/2}\varpi\big\{\ha_\xi[\k(x,\zeta)]\haa(\xi,\zeta)\big\}\Delta_\G(z)^{1/2}\varpi\{\ha_{\xi\zeta}[\aa(x,z)]\}\Omega(xz,\xi\zeta)\\
&=\pi[\k(z,\zeta)](V_\zeta U_z\Omega)(x,\xi)\,.
\end{aligned}
$$
}
\end{Example}

\section{The twisted action attached to a covariant structure}\label{vizigot}

\medskip
Let us set for $x,y\in\G$ and $\xi,\eta\in\hG$
\begin{equation}\label{zmeul}
\la_{\!(x,\xi)}:=\ha_\xi\circ\a_x\,,
\end{equation}
\begin{equation}\label{ballaur}
\laa\big((x,\xi),(y,\eta)\big):=
\ha_\xi\!\[\k(x,\eta)\]\haa(\xi,\eta)\,\ha_{\xi\eta}[\aa(x,y)]\,.
\end{equation}

\begin{Proposition}\label{sconx}
$(\la,\laa)$ is a measurable twisted action of $\,\G\times\hG$ on $\A$\,. If the two twisted actions $(\a,\aa)$ and $(\ha,\haa)$ are continuous, then $(\la,\laa)$ is continuous.
\end{Proposition}

\begin{proof}
Using the assumptions and relations as $\Psi\circ\ad_B=\ad_{\Psi(B)}\circ\Psi$ and $\ad_A\circ\ad_B=\ad_{AB}$ one computes
$$
\begin{aligned}
\la_{(x,\xi)}\circ\la_{(y,\eta)}&=\ha_\xi\circ\a_x\circ\ha_\eta\circ\a_y\\
&=\ha_\xi\circ\ad_{\k(x,\eta)}\circ\ha_\eta\circ\a_x\circ\a_y\\
&=\ad_{\ha_\xi[\k(x,\eta)]}\circ\ha_\xi\circ\ha_\eta\circ\a_x\circ\a_y\\
&=\ad_{\ha_\xi[\k(x,\eta)]}\circ\ad_{\haa(\xi,\eta)}\circ\ha_{\xi\eta}\circ\ad_{\aa(x,y)}\circ\a_{xy}\\
&=\ad_{\ha_\xi[\k(x,\eta)]}\circ\ad_{\haa(\xi,\eta)}\circ\ad_{\ha_{\xi\eta}[\aa(x,y)]}\circ\ha_{\xi\eta}\circ\a_{xy}\\
&=\ad_{\laa\big((x,\xi),(y,\eta)\big)}\circ\la_{(xy,\xi\eta)}\,.
\end{aligned}
$$
One computes with a huge pacience
$$
\begin{aligned}
&\laa\big((x,\xi),(y,\eta)\big)\laa\big((xy,\xi\eta),(z,\zeta)\big)\\
=\ &\ha_\xi\!\[\k(x,\eta)\]\haa(\xi,\eta)\,\ha_{\xi\eta}[\aa(x,y)]\,\ha_{\xi\eta}[\k(xy,\zeta)]\,\haa(\xi\eta,\zeta)\,\ha_{\xi\eta\zeta}[\aa(xy,z)]\\
=\ &\ha_\xi\big\{\k(x,\eta)\ha_{\eta}[\aa(x,y)\k(xy,\zeta)]\big\}\haa(\xi,\eta)\,\haa(\xi\eta,\zeta)\,\ha_{\xi\eta\zeta}[\aa(xy,z)]\\
\overset{(\ref{spoarca})}{=}&\ha_\xi\big\{\k(x,\eta)\ha_{\eta}[\a_x(\k(y,\zeta))\k(x,\zeta)\ha_\zeta(\aa(x,y))]\big\}\,\haa(\xi,\eta)\,\haa(\xi\eta,\zeta)\,\ha_{\xi\eta\zeta}[\aa(xy,z)]\\
=\ &\ha_\xi\big\{\k(x,\eta)\ha_{\eta}[\a_x(\k(y,\zeta))\k(x,\zeta)]\big\}\,(\ha_\xi\circ\ha_\eta\circ\ha_\zeta)[\aa(x,y)]\,\haa(\xi,\eta)\,\haa(\xi\eta,\zeta)\,\ha_{\xi\eta\zeta}[\aa(xy,z)]\\
=\ &\ha_\xi\big\{\k(x,\eta)\ha_{\eta}[\a_x(\k(y,\zeta))\k(x,\zeta)]\big\}\,\haa(\xi,\eta)\,\haa(\xi\eta,\zeta)\,\ha_{\xi\eta\zeta}[\aa(x,y)]\,\ha_{\xi\eta\zeta}[\aa(xy,z)]\\
=\ &\ha_\xi\big\{\k(x,\eta)\ha_{\eta}[\a_x(\k(y,\zeta))\k(x,\zeta)]\big\}\,\ha_\xi[\haa(\eta,\zeta)]\,\haa(\xi,\eta\zeta)\,\ha_{\xi\eta\zeta}\big\{\a_x[\aa(y,z)]\aa(x,yz)]\big\}\,.\\
\end{aligned}
$$
On the other hand
$$
\begin{aligned}
&\la_{(x,\xi)}\!\[\laa\big((y,\eta),(z,\zeta)\big)\]\laa\big((x,\xi),(yz,\eta\zeta)\big)\\
=\ \ &(\ha_\xi\circ\a_x)\!\big\{\ha_\eta[\k(y,\zeta)]\haa(\eta,\zeta)\ha_{\eta\zeta}[\aa(y,z)]\big\}\,\ha_\xi[\k(x,\eta\zeta)]\haa(\xi,\eta\zeta)\,\ha_{\xi\eta\zeta}[\aa(x,yz)]\\
=\ \ &\ha_\xi\big[\a_x\{\ha_\eta[\k(y,\zeta)]\haa(\eta,\zeta)\ha_{\eta\zeta}[\aa(y,z)]\}\k(x,\eta\zeta)\big]\haa(\xi,\eta\zeta)\,\ha_{\xi\eta\zeta}[\aa(x,yz)]\\
=\ \ &\ha_\xi\big\{(\a_x\circ\ha_\eta)[\k(y,\zeta)]\,\a_x[\haa(\eta,\zeta)]\,(\a_x\circ\ha_{\eta\zeta})[\aa(y,z)]\k(x,\eta\zeta)\big\}\haa(\xi,\eta\zeta)\,\ha_{\xi\eta\zeta}[\aa(x,yz)]\\
\overset{(\ref{dihanie})}{=}&\ha_\xi\big\{\k(x,\eta)(\ha_\eta\circ\a_x)[\k(y,\zeta)]\k(x,\eta)^*\,\a_x[\haa(\eta,\zeta)]\,\k(x,\eta\zeta)(\ha_{\eta\zeta}\circ\a_{x})[\aa(y,z)]\big\}\haa(\xi,\eta\zeta)\,\ha_{\xi\eta\zeta}[\aa(x,yz)]\\
=\ &\ha_\xi\big\{\k(x,\eta)(\ha_\eta\circ\a_x)[\k(y,\zeta)]\k(x,\eta)^*\,\a_x[\haa(\eta,\zeta)]\,\k(x,\eta\zeta)\big\}(\ha_\xi\circ\ha_{\eta\zeta})\{\a_{x}[\aa(y,z)]\}\haa(\xi,\eta\zeta)\,\ha_{\xi\eta\zeta}[\aa(x,yz)]\\
=\ &\ha_\xi\big\{\k(x,\eta)(\ha_\eta\circ\a_x)[\k(y,\zeta)]\k(x,\eta)^*\,\a_x[\haa(\eta,\zeta)]\,\k(x,\eta\zeta)\big\}\haa(\xi,\eta\zeta)\ha_{\xi\eta\zeta}\{\a_{x}[\aa(y,z)]\}\,\ha_{\xi\eta\zeta}[\aa(x,yz)]\\
\overset{(\ref{poarca})}{=}&\ha_\xi\big\{\k(x,\eta)(\ha_\eta\circ\a_x)[\k(y,\zeta)]\ha_\eta[\k(x,\zeta)]\haa(\eta,\zeta)\big\}\haa(\xi,\eta\zeta)\ha_{\xi\eta\zeta}\{\a_{x}[\aa(y,z)]\}\,\ha_{\xi\eta\zeta}[\aa(x,yz)]\\
=\ &\ha_\xi\big\{\k(x,\eta)(\ha_\eta\circ\a_x)[\k(y,\zeta)]\ha_\eta[\k(x,\zeta)]\big\}\ha_\xi[\haa(\eta,\zeta)]\haa(\xi,\eta\zeta)\ha_{\xi\eta\zeta}\{\a_{x}[\aa(y,z)]\}\,\ha_{\xi\eta\zeta}[\aa(x,yz)]\,,\\
\end{aligned}
$$
the two expressions coincide and thus the $2$-cocycle condition is verified. The normalization of $\laa$ is obvious. 

The continuity and the measurability are easy.
\qed
\end{proof}

\begin{Remark}\label{strufat}
{\rm Relation (\ref{vrajitroaca}) can be rephrased, also using (\ref{dihanie})
\begin{equation}\label{vrajitroasca}
\k(x,\xi)\,(\ha_\xi\circ \a_x)[\k(y,\eta)]\left\{\ha_\xi\!\[\k(x,\eta)\]\haa(\xi,\eta)\,\ha_{\xi\eta}\!\[\aa(x,y)\]\right\}=\a_x[\k(y,\xi)^*]\,\aa(x,y)\,\a_{xy}[\haa(\xi,\eta)]\,.
\end{equation}
The r.h.s. of (\ref{vrajitroasca}) defines a $2$-cocycle $\raa$ on $\G\times\hG$ with respect to $\ra_{\!(x,\xi)}:=a_x\circ\ha_\xi$ and (\ref{vrajitroaca}) can be rewritten
\begin{equation}\label{vrajitroasco}
\k(x,\xi)\,\la_{\!(x,\xi)}[\k(y,\eta)]\laa\big((x,\xi),(y,\eta)\big)\k(xy,\xi\eta)^*=\raa\big((x,\xi),(y,\eta)\big)\,.
\end{equation} 
Relations (\ref{dihanie}) and (\ref{vrajitroasco}) tell that the twisted actions $(\la,\laa)$ and $(\ra,\raa)$ are exterior equivalent (Remark \ref{duniath} and \cite{PR1}) through the $1$-cochain $\k$\,. Rephrasings in terms of the group ${\sf H}':=\hG\times\G$\,, based on Remark \ref{ghiveci}, are left to the reader.
}
\end{Remark}

\medskip
\begin{Remark}\label{mamaliga}
{\rm Now that we have introduced all the notations, it may be useful for the reader to recall the definition of a {\it covariant structure} $\{(\A,\kappa),(\a,\aa),(\ha,\haa)\}$: It is defined by a twisted action $(\a,\aa)$ of the group $\G$\,, a twisted action $(\ha,\haa)$ of the group $\hG$ and a normalized strictly continuous map $\kappa:\G\times\hG\to\U\M(\A)$ such that for all $X,Y\in\G\times\hG$
\begin{equation*}\label{mismas}
\ra_{\!X}=\ad_{\k(X)}\circ\la_{\!X}\quad{\rm and}\quad\k(X)\,\la_{\!X}[\k(Y)]\laa\big(X,Y\big)\k(XY)^*=\raa\big(X,Y\big)\,. 
\end{equation*}
Using a notation of Remark \ref{duniath}, this can be written $(\la,\laa)\overset{\k}{\sim}(\ra,\raa)$\,.
}
\end{Remark}

\medskip
\begin{Proposition}\label{eredpizduin}
There are one-to-one correspondences between:
\begin{enumerate}
\item
Covariant morphisms $(\B,r,u,v)$ of the covariant structure $\left\{(\A,\kappa),(\a,\aa),(\ha,\haa)\right\}$ (cf. Def. \ref{stirba})\,.
\item
Covariant morphisms $(\B,r,w)$ of the twisted $C^*$-dynamical system $(\A,\la,\laa)$ with group ${\sf H}:=\G\times\widetilde\G$\,.
\item
Covariant morphisms $(\B,r,w')$ of the twisted $C^*$-dynamical system $(\A,\ra,\raa)$ with group ${\sf H}:=\G\times\hG$\,.
\end{enumerate}
\end{Proposition}

\begin{proof}
If $(\B,r,u,v)$ is given, one defines 
\begin{equation}\label{zaluga}
w:\G\times\hG\to\U\M(\B)\,,\quad\,w(x,\xi):=v_\xi u_x=r[\k(x,\xi)^*]u_x v_\xi\,.
\end{equation}
We show that $(\B,r,w)$ is a covariant morphism of $(\A,\la,\laa)$\,. If $(x,\xi),(y,\eta)\in\G\times\hG$ one has
$$
\begin{aligned}
w(x,\xi)w(y,\eta)&=v_\xi u_x v_\eta u_y\\
&=v_\xi r[\k(x,\eta)]v_\eta u_x u_y \\
&=r\!\left\{\ha_\xi\!\[\k(x,\eta)\]\right\}v_\xi v_\eta u_x u_y\\
&=r\!\left\{\ha_\xi\!\[\k(x,\eta)\]\right\}r[\haa(\xi,\eta)]v_{\xi\eta}\,r[\aa(x,y)]u_{xy}\\
&=r\!\left\{\ha_\xi\!\[\k(x,\eta)\]\right\}r[\haa(\xi,\eta)]\,r\{\ha_{\xi\eta}[\aa(x,y)]\}\,v_{\xi\eta}u_{xy}\\
&=r\big\{\ha_\xi\!\[\k(x,\eta)\]\haa(\xi,\eta)\,\ha_{\xi\eta}\!\[\aa(x,y)\]\!\big\}\,w(xy,\xi\eta)\\
&=r\!\[\laa\big((x,y),(\xi,\eta)\big)\]w\big((x,\xi)(y,\eta)\big)\,.
\end{aligned}
$$
On the other hand, for $(x,\xi)\in\G\times\hG$ and $A\in\A$ one gets
$$
\begin{aligned}
w(x,\xi)r(A)w(x,\xi)^*&=v_\xi u_x r(A)u_x^*v_\xi^*\\
&=v_\xi r[\a_x(A)]v_\xi^*\\
&=r\!\left\{\ha_\xi\big[\a_x(A)\big]\right\}\\
&=r\!\[\la_{\!(x,\xi)}(A)\].
\end{aligned}
$$

Now assume that $(\B,r,w)$ is a covariant representation of the twisted $C^*$-dynamical system $(\A,\la,\laa)$\,. Defining $\,u:\G\to\U\M(\B)$ and $v:\hG\to\U\M(\B)$ by
\begin{equation}\label{gheonoaie}
\ u_x:=w(x,\varepsilon)\,,\quad\ v_\xi:=w({\sf e},\xi)
\end{equation}
one gets a quadruple $(\B,r,u,v)$ satisfying the conditions specified at $1$\,. We leave the easy verifications to the reader. Among others one uses the relations
$$
\laa\big((x,\varepsilon),(y,\varepsilon)\big)=\aa(x,y)\,,\ \quad\laa\big(({\sf e},\xi),({\sf e},\eta)\big)=\haa(\xi,\eta)\,,
$$
$$
\laa\big((x,\varepsilon),({\sf e},\eta)\big)=\k(x,\eta)\,,\ \quad\laa\big(({\sf e},\xi),(y,\varepsilon)\big)=1\,.
$$
So we made explicit the correspondence between 1 and 2\,. The correspondence between 1 and 3 is analogous; just put 
\begin{equation*}\label{drac}
w'(x,\xi):=u_x v_\xi\quad{\rm for}\ (x,\xi)\in\G\times\hG\,. 
\end{equation*}
\qed
\end{proof}

\section{The bi-product of a covariant structure}\label{vizigott}

\medskip
\begin{Definition}\label{djin}
Let $\left\{(\A,\k),(\a,\aa),(\ha,\haa)\right\}$ be a given covariant structure. {\rm A (twisted crossed) bi-product} is a universal covariant morphism $\(\C,\iota_\A,\iota_\G,\iota_{\hG}\)$\,. Universality means that if $(\B,r,u,v)$ is another covariant morphism, there exists a unique non-degenerate morphism $s:\C\to\M(\B)$ such that 
\begin{equation}\label{ghiaur}
u=s\circ\iota_\G\,,\quad v=s\circ\iota_{\hG}\,,\quad r=s\circ\iota_\A\,.
\end{equation}
\end{Definition}

Rather often we will call {\it bi-product} only the $C^*$-algebra $\C$, especially when the mappings $\(\iota_\A,\iota_\G,\iota_{\hG}\)$ are obvious or not relevant. It could be denoted generically by $\C\equiv\A_{(\a,\ha)}^{(\aa,\haa)}$, but it also depends on $\k$\,; its existence and (essential) uniqueness will be proved now.

\medskip
\begin{Proposition}\label{simurg}
Every covariant structure possesses a (twisted crossed) bi-product, that is unique up to a canonical isomorphism.
\end{Proposition}

\begin{proof}
By an easy abstract argument, if a bi-product exists, it is unique up to a canonical isomorphism. The meaning of this and the proof are the standard ones. 

To prove existence, we rely on Proposition \ref{eredpizduin} and on the universality of the usual twisted crossed products. If $\left\{(\A,\k),(\a,\aa),(\ha,\haa)\right\}$ is a covariant structure, we construct as above the twisted $C^*$-dynamical system $(\A,\la,\laa)$ with group $\G\times\hG$\,. Let $\(\C,\iota_\A,\iota_{\G\times\hG}\)$ be a corresponding twisted crosed product. Recalling (\ref{gheonoaie}) we set
\begin{equation}\label{dragonel}
\iota_\G:\G\rightarrow\U\M(\C)\,,\quad\ \iota_\G(x):=\iota_{\G\times\hG}(x,\varepsilon)\,,
\end{equation}
\begin{equation}\label{dragonez}
\iota_{\hG}:\hG\rightarrow\U\M(\C)\,,\quad\ \iota_{\hG}(\xi):=\iota_{\G\times\hG}({\sf e},\xi)\,.
\end{equation}
From Proposition \ref{eredpizduin} we already know that $\(\C,\iota_\A,\iota_\G,\iota_{\hG}\)$ is a covariant morphism; one must show its universality.
So let $(\B,r,u,v)$ be another covariant morphism and let us define $w$ as in (\ref{zaluga})\,. Since $(\B,r,w)$ is a covariant morphism of $(\A,\la,\laa)$\,, there exists a unique $C^*$-algebraic morphism $s:\C\rightarrow\M(\B)$ such that
\begin{equation}\label{fonfanel}
w=s\circ\iota_{\G\times\hG}\,,\quad r=s\circ\iota_\A\,.
\end{equation}
Then we have
$$
(s\circ\iota_\G)(x)=s\!\[\iota_\G(x)\]=s\!\[\iota_{\G\times\hG}(x,\varepsilon)\]=w(x,\varepsilon)=u(x)
$$
and
$$
(s\circ\iota_{\hG})(\xi)=s\!\[\iota_{\hG}(\xi)\]=s\!\[\iota_{\G\times\hG}({\sf e},\xi)\]=w({\sf e},\xi)=v(\xi)
$$
and we are done.
\qed
\end{proof}

Relying on the twisted actions $(\la,\laa)$ and $(\ra,\raa)$ we get new $C^*$-algebras
\begin{equation*}\label{mumapadurii}
\A_{\la}^{\laa}:=\A\!\rtimes_{\la}^{\laa}\!(\G\times\hG)\quad{\rm with\ laws}\ \big(\overrightarrow{\#},^{\overrightarrow{\#}}\big)
\end{equation*}
and
\begin{equation*}\label{tutapadurii}
\A_{\ra}^{\raa}:=\A\!\rtimes_{\ra}^{\raa}\!(\G\times\hG)\quad{\rm with\ laws}\ \big(\overleftarrow\#,^{\overleftarrow{\#}}\big)\,.
\end{equation*}
They can be viewed as concrete realizations of the bi-product $C^*$-algebra $\A_{(\a,\ha)}^{(\aa,\haa)}$\,. Of course they are isomorpic, being defined by exterior equivalent twisted actions, cf. Remarks \ref{strufat} and \ref{duniath}.
It will be convenient to regard them as the enveloping $C^*$-algebras of the corresponding $L^1$ Banach $^*$-algebras (but the abstract universal approach could also be adopted). At the $L^1$-level the isomorphism is given by $\overrightarrow F\to \overrightarrow F\kappa^*$\,. For further use, we record here the composition laws on $\A_{\la}^{\laa}$
\begin{equation}\label{gulas}
\(\overrightarrow{F}\,\overrightarrow{\#}\,\overrightarrow{G}\)\!(x,\xi)=\!\int_\G\!\int_{\hG}\!dyd\eta \overrightarrow{F}(y,\eta)(\ha_\eta\circ\a_y)\!\[\overrightarrow{G}(y^{-1}x,\eta^{-1}\xi)\]\ha_\eta[\kappa(y,\eta^{-1}\xi)]\haa(\eta,\eta^{-1}\xi)\ha_\xi[\aa(y,y^{-1}x)]\,,
\end{equation}
\begin{equation}\label{gulash}
(\overrightarrow{F}^{\overrightarrow{\#}})(x,\xi)=\Delta_\G(x^{-1})\Delta_{\hG}(\xi^{-1})\aa(x,x^{-1})^*\haa(\xi,\xi^{-1})^*\,\ha_\xi[\k(x,\xi^{-1})^*]\,(\ha_\xi\circ\a_x)\!\[\overrightarrow{F}(x^{-1},\xi^{-1})^*\]
\end{equation}
and on $\A_{\ra}^{\raa}$
\begin{equation}\label{guleras}
\(\overleftarrow{F}\,\overleftarrow{\#}\,\overleftarrow{G}\)\!(x,\xi)=\!\int_\G\!\int_{\hG}\!dyd\eta \overleftarrow{F}(y,\eta)(\a_y\circ\ha_\eta)\!\[\overleftarrow{G}(y^{-1}x,\eta^{-1}\xi)\]\a_y[\kappa(y^{-1}x,\eta)^*]\aa(y,y^{-1}x)\a_x[\haa(\eta,\eta^{-1}\xi)],
\end{equation}
\begin{equation}\label{gulerash}
(\overleftarrow{F}^{\overleftarrow{\#}})(x,\xi)=\Delta_\G(x^{-1})\Delta_{\hG}(\xi^{-1})\haa(\xi,\xi^{-1})^*\aa(x,x^{-1})^*\,\a_x[\k(x^{-1},\xi)]\,(\a_x\circ\ha_\xi)\!\[\overleftarrow{F}(x^{-1},\xi^{-1})^*\].
\end{equation}

By using Remark \ref{ghiveci}, one generates other two twisted actions of the group $\hG\times\G$ in $\A$ as well as other two twisted crossed product $C^*$-algebras isomorphic to the previous ones. They can also be seen as concrete realizations of the bi-product $\A_{(\a,\aa)}^{(\ha,\haa)}$\,.

\medskip
The next Corollary is now obvious. Similar statements hold at the level of (covariant) morphisms.

\medskip
\begin{Corollary}\label{eredpisuin}
There are one-to-one correspondences between:
\begin{enumerate}
\item
Covariant representations $(\H,\pi,U,V)$ of the covariant structure $\left\{(\A,\k),(\a,\aa),(\ha,\haa)\right\}$\,.
\item
Covariant representations $(\H,\pi,W)$ of the twisted $C^*$-dynamical system $(\A,\la,\laa)$ with group $\G\times\hG$\,.
\item
Covariant representations $(\H,\pi,W')$ of the twisted $C^*$-dynamical system $(\A,\ra,\raa)$ with group $\G\times\hG$\,.
\item
Non-degenerate representations of the bi-product $\A_{(\a,\ha)}^{(\aa,\haa)}$\,.
\item
Non-degenerate representations of the $C^*$-algebra $\A_{\la}^{\laa}$\,.
\item
Non-degenerate representations of the $C^*$-algebra $\A_{\ra}^{\raa}$\,.
\end{enumerate}
\end{Corollary}

\medskip
\begin{Example}\label{placinta}
{\rm In Example \ref{papaleasca}, given a representation $\varpi$ of the $C^*$-algebra $\A$ in the Hilbert space $\H$\,, we constructed the corresponding induced covariant representation $(\pi,U,V)$ of the covariant structure $\{(\A,\k),(\a,\aa),(\ha,\haa)\}$ in the Hilbert space $\mathscr H=L^2(\G\times\hG;\H)$\,. Applying to it the construction given in the proof of Proposition \ref{eredpizduin}, one gets exactly the induced covariant representation \cite[Def. 3.10]{PR1} $(\mathscr H,\pi,W)$ of the twisted $C^*$-dynamical system $(\A,\la,\laa)$ with group $\G\times\hG$ attached to the initial $\varpi$\,. 
}
\end{Example}

\section{First and second generation twisted crossed products}\label{vipsgrot}

Let $\left\{(\A,\k),(\a,\aa),(\ha,\haa)\right\}$ be a given covariant structure. To associate to it another (particular) covariant structure $\left\{(\A_\a^\aa,{\sf k}),(\b,\bb),(\hb,\hbb)\right\}$\,, we first set $\A_\a^\aa:=\A\!\rtimes_\a^\aa\!\G$ with algebraic laws $(\diamond,^\diamond)$\,. Also set
\begin{equation}\label{imbar}
{\sf k}:\G\times\hG\to \mathcal U\mathcal M(\A_\a^\aa)\,,\quad\ {\sf k}(x,\xi):=\delta_{\sf e}\otimes\k(x,\xi)\,.
\end{equation}
From (\ref{ossiriand}) and (\ref{dorthonion}) and from $\p\!\cdot\!\p_{\A_\a^\aa}\,\le\,\p\!\cdot\!\p_1$ it follows easily that ${\sf k}$ is strictly continuous.

For each $\xi\in\hG$ we define $\hb_\xi:L^1(\G;\A)\rightarrow L^1(\G;\A)$ by
\begin{equation}\label{hobbit}
\[\hb_\xi(f)\]\!(y):=\ha_\xi[f(y)]\k(y,\xi)^*\,,
\end{equation}
while for $\xi,\eta\in\hG$\,, based on the preparations made in section \ref{fourtundoy}, we set
\begin{equation}\label{thalos}
\hbb(\xi,\eta):=\delta_{\sf e}\otimes\haa(\xi,\eta)\in\mathcal U\mathcal M(\A_\a^\aa)\,.
\end{equation}

\begin{Proposition}\label{trol}
The pair $(\hb,\hbb)$ defines a measurable twisted action of $\,\hG$ on $\A_\a^\aa$\,. If $\,(\ha,\haa)$ is continuous, then $(\hb,\hbb)$ is also continuous.
\end{Proposition}

\begin{proof} 
1. We need to prove that $\hb_\xi$ is an automorphism of $\A\!\rtimes_\a^\aa\!\G$\,. We only show that $\hb_\xi:L^1(\G;\A)\rightarrow L^1(\G;\A)$ is a $^*$-isomorphism for the twisted crossed product structure; then the extension to the full twisted crossed product is automatic. Clearly $\hb_\xi$ is well-defined and invertible and one has $\hb_\varepsilon={\rm id}$\,.

For the product, using the definitions, (\ref{dihanie}) and (\ref{spoarca}) one gets
$$
\begin{aligned}
\[\hb_\xi(f)\diamond\hb_\xi(g)\]\!(x)&=\int_\G\!dy\[\hb_\xi(f)\]\!(y)\,\a_y\!\left\{\[\hb_\xi(g)\]\!(y^{-1}x)\right\}\aa(y,y^{-1}x)\\
&=\int_\G\!dy\,\ha_\xi[f(y)]\,\k(y,\xi)^*\,(\a_y\circ\ha_\xi)[g(y^{-1}x)]\,\a_y[\k(y^{-1}x,\xi)^*]\aa(y,y^{-1}x)\\
&=\int_\G\!dy\,\ha_\xi[f(y)]\,(\ha_\xi\circ\a_y)[g(y^{-1}x)]\,\k(y,\xi)^*\,\a_y[\k(y^{-1}x,\xi)^*]\aa(y,y^{-1}x)\\
&=\int_\G\!dy\,\ha_\xi[f(y)]\,\ha_\xi\!\left\{\a_y\,[g(y^{-1}x)]\right\}\ha_\xi[\aa(y,y^{-1}x)]\k(x,\xi)^*\\
&=\ha_\xi\!\(\int_\G\!dy\,f(y)\,\a_y\!\left[g(y^{-1}x)\right]\aa(y,y^{-1}x)\)\k(x,\xi)^*=\[\hb_\xi(f\diamond_\a^\aa g)\]\!(x)\,.
\end{aligned}
$$

For the involution, by (\ref{dihanie}) and (\ref{poarca}):
$$
\begin{aligned}
\[\hb_\xi(f)\]^{\diamond}\!(x)&=\Delta_\G(x^{-1})\,\aa(x,x^{-1})^*\,\a_{x}\!\[\hb_\xi(f)(x^{-1})\]^*\\
&=\Delta_\G(x^{-1})\,\aa(x,x^{-1})^*\,\a_{x}\!\left\{\ha_\xi\!\[f(x^{-1})\]\k(x^{-1},\xi)^*\right\}^*\\
&=\Delta_\G(x^{-1})\,\aa(x,x^{-1})^*\,\a_x[\k(x^{-1},\xi)]\,\a_{x}\!\left\{\ha_\xi\!\[f(x^{-1})\]\right\}^*\\
&=\Delta_\G(x^{-1})\,\aa(x,x^{-1})^*\,\a_x[\k(x^{-1},\xi)]\,\k(x,\xi)\,\ha_{\xi}\!\left\{\a_x\!\[f(x^{-1})\]\right\}^*\k(x,\xi)^*\\
&=\Delta_\G(x^{-1})\,\ha_\xi[\aa(x,x^{-1})^*]\,\ha_\xi\!\left\{\a_{x}\!\[f(x^{-1})^*\]\right\}\k(x,\xi)^*\\
&=\ha_\xi\!\left\{\Delta_\G(x^{-1})\aa(x,x^{-1})^*\a_{x}\!\[f(x^{-1})^*\]\right\}\k(x,\xi)^*\\
&=\ha_\xi\!\[f^{\diamond}(x)\]\k(x,\xi)^*=\[\hb_\xi(f^{\diamond})\]\!(x)\,.
\end{aligned}
$$

\medskip
2. For $\xi,\eta\in\G$ we show that $\hb_\xi\circ\hb_\eta={\sf ad}^\diamond_{\hbb(\xi,\eta)}\!\circ\hb_{\xi\eta}$\,. One computes for $x\in\G$ and $f\in L^1(\G;\A)$
$$
\begin{aligned}
\[\(\hb_\xi\circ\hb_\eta\)\!(f)\]\!(x)&=\ha_\xi\!\[\hb_\eta(f)(x)\]\k(x,\xi)^*\\
&=(\ha_\xi\circ\ha_\eta)[f(x)]\ha_\xi[\k(x,\eta)^*]\k(x,\xi)^*\\
&=\haa(\xi,\eta)\,\ha_{\xi\eta}[f(x)]\,\haa(\xi,\eta)^*\ha_\xi[\k(x,\eta)^*]\k(x,\xi)^*\\
&=\haa(\xi,\eta)\,\ha_{\xi\eta}[f(x)]\,\k(x,\xi\eta)^*\a_x[\haa(\xi,\eta)^*]\\
&=\haa(\xi,\eta)\[\hb_{\xi\eta}(f)\]\!(x)\,\a_x[\haa(\xi,\eta)^*]\\
&=\(\hbb(\xi,\eta)\diamond\[\hb_{\xi\eta}(f)\]\diamond\hbb(\xi,\eta)^\diamond\)\!(x)\,.
\end{aligned}
$$
We used (\ref{poarca}); to justify the last equality use (\ref{himlad}), (\ref{ladros})\,.

\medskip
3. Now we show that $\hbb$ is a $2$-cocycle with respect to $\hb$\,. The normalization is clear.
To check the $2$-cocycle identity, from the definition of $\hbb$\,, (\ref{mistymountains}) and the fact (following from (\ref{ossiriand}) and (\ref{dorthonion})) that $\hb_\xi(\delta_{\sf e}\otimes m)=\delta_{\sf e}\otimes\ha_\xi(m)$ one gets
$$
\begin{aligned}
\hbb(\xi,\eta)\diamond\hbb(\xi\eta,\zeta)&=\[\delta_{\sf e}\otimes \haa(\xi,\eta)\]\diamond\[\delta_{\sf e}\otimes \haa(\xi\eta,\zeta)\]\\
&=\delta_{\sf e}\otimes\[\haa(\xi,\eta)\haa(\xi\eta,\zeta)\]\\
&=\delta_{\sf e}\otimes\[\ha_{\xi}\!\(\haa(\eta,\zeta)\)\haa(\xi,\eta\zeta)\]\\
&=\left\{\delta_{\sf e}\otimes\ha_{\xi}\!\[\haa(\eta,\zeta)\]\right\}\diamond\[\delta_{\sf e}\otimes \haa(\xi,\eta\zeta)\]\\
&=\hb_\xi\[\delta_{\sf e}\otimes \haa(\eta,\zeta)\]\diamond\[\delta_{\sf e}\otimes \haa(\xi,\eta\zeta)\]\\
&=\hb_\xi\!\[\hbb(\eta,\zeta)\]\diamond\hbb(\xi,\eta\zeta)\,.
\end{aligned}
$$

\medskip
4. Assuming now that $(\ha,\haa)$ is continuous, we are going to show that $(\tilde\b,\tilde\beta)$ is continuous. We indicate the rather straightforward arguments, because changes of norms are involved.

To show that $\tilde\b$ is strongly continuous, we estimate for $f=\varphi\otimes A$ in the dense algebraic tensor product $L^1(\G)\odot\A$
$$
\begin{aligned}
\p\!\tilde\b_\eta(f)-\tilde\b_\xi(f)\!\p_{\A_\a^\aa}\,\le\,\p\!\tilde\b_\eta(f)-\tilde\b_\xi(f)\!\p_1\,\le\int_\G\!dx\,|\varphi(x)|\,\big\|\ha_\eta(A)\k(x,\eta)^*-\ha_\xi(A)\k(x,\xi)^*\,\big\|_\A\,.
\end{aligned}
$$
By the Dominated Convergence Theorem, the integrability of $\varphi$ and the bound 
$$
\big\|\ha_\eta(A)\k(x,\eta)^*-\ha_\xi(A)\k(x,\xi)^*\big\|_\A\le 2\p\!A\!\p_\A\,,
$$ 
it is enough to prove that for $x\in\G$ the integrant converges to zero when $\eta\to\xi$\,, which is trivial since $\ha$ is strongly continuous and $\k(x,\cdot)$ is strictly continuous. 

Then, using (\ref{ossiriand})
$$
\begin{aligned}
\big\|\tilde\beta(\xi',\eta')\diamond f-\tilde\beta(\xi,\eta)\diamond f\big\|_{\A_\a^\aa}&\le\big\|\,[\delta_{\sf e}\otimes\tilde\alpha(\xi',\eta')]\diamond f-\delta_{\sf e}\otimes\tilde\alpha(\xi,\eta)]\diamond f\,\big\|_1\\
&\le\int_\G\!dx\,|\varphi(x)|\!\p\!\tilde\alpha(\xi',\eta')A- \tilde\alpha(\xi,\eta)A\!\p_\A.
\end{aligned}
$$
Once again it follows that this converges to zero if $(\xi',\eta')\to(\xi,\eta)$, using the Dominated Convergence Theorem, the integrability of $\varphi$ and the fact that $\tilde\aa$ is strictly continuous. Multiplying with $f$ to the left is treated similarly.

\medskip
5. By using the definition of strong or strict measurability, one is lead to show that a map $h$ defined from a Hausdorff, second countable locally space ${\sf X}$ endowed with a Radon measure $\mu$ to a separable Banach space $\mathscr B$ is measurable. The next criterion \cite[App. B]{Wi} reduces this to an easier continuity issue:

{\it A function $h:{\sf X}\to\mathscr B$ is measurable if and only if for any compact set $K\subset{\sf X}$ and any $\epsilon>0$\,, there exists a subset $K^\prime\subset K$ such that $\mu(K\setminus K^{\prime})\le\epsilon$ and the restriction $h|_{K^\prime}$ is continuous.}

Now our measurable case follows rather easily from this and from the previous point 4. To illustrate the case of the action $\hb$\,, we start once again with vectors of the form $f=\varphi\otimes A$\,, where $\varphi\in L^1(\G)$ and $A\in\A$\,. Pick a compact set $K\subset\hG$ and a strictly positive number $\epsilon$\,; for some subset $K'$ of $K$ for which the Haar measure of $K\setminus K'$ is smaller than $\epsilon$\,, the restrictions to $K'$ of the maps $\xi\to\ha_\xi(A)$ and $\xi\to\k(x,\xi)$ are continuous for all $x\in\G$\,. By the argument above, the restriction to $K'$ of the map $\xi\mapsto\hb_\xi(\varphi\otimes A)$ is continuous. This and linearity show that the map $\xi\mapsto\hb_\xi(f)$ is measurable for any vector $f$ belonging to the dense subset $L^1(\G)\odot\A$ of $\A_\a^\aa$\,. Passing to an arbitrary vector is easy by density, applying a $\delta/3$ trick and the criterion again. The strict measurability of $\hbb$ is treated similarly.
\qed
\end{proof}

We define now the twisted action of $\G$ on the twisted crossed product. First, for $x\in\G$\,, let us set
\begin{equation*}\label{lothlann}
\lambda_x:=\delta_x\otimes 1\in\mathcal U\mathcal M(\A_\a^\aa)\,.
\end{equation*}
Deducing strict continuity or measurability from similar properties of the twisted action $(\a,\aa)$ is straightforward, if one takes (\ref{ossiriand}) and (\ref{dorthonion}) into consideration. 
A computation relying on (\ref{ossiriand}) leads to the covariance condition
\begin{equation*}\label{dirnen}
\hb_\xi(\lambda_x)=[\delta_{\sf e}\otimes\k(x,\xi)^*]\diamond\lambda_x={\sf k}(x,\xi)^\diamond\diamond\lambda_x\,,\quad\ \forall\,x\in\G\,,\,\xi\in\hG\,.
\end{equation*}
Along the lines of Example \ref{besteleala}, define $\b:\G\rightarrow\Aut(\A_\a^\aa)$ by
\begin{equation*}\label{pelori}
\b_x(f):=\ad^\diamond_{\lambda_x}(f)=\lambda_x\diamond f\diamond\lambda_x^\diamond
\end{equation*}
and $\bb:\G\times\G\rightarrow\mathcal U\mathcal M(\A_\a^\aa)$ by
\begin{equation*}\label{avathar}
\bb(x,y):=\lambda_x\diamond\lambda_y\diamond\lambda_{xy}^\diamond=\delta_{\sf e}\otimes \aa(x,y)\,.
\end{equation*}

All the calculations above conclude by

\medskip
\begin{Theorem}\label{ainulindale}
If $\,\left\{(\A,\k),(\a,\aa),(\ha,\haa)\right\}$ is a given measurable (resp. continuous) covariant structure, then 
$\big\{(\A\!\rtimes_\a^\aa\!\G,{\sf k}),(\b,\beta),(\tilde{\b},\tilde{\beta})\big\}$ is a measurable (resp. continuous) $\G$-particular covariant structure.
\end{Theorem}

\medskip
Starting with the same covariant structure $\left\{(\A,\k),(\a,\aa),(\ha,\haa)\right\}$\,, one can also construct a $\hG$-particular covariant structure $\left\{\big(\A_\ha^{\haa},\tilde{\sf k}\big),(\mathrm c,\gamma),(\tilde\c,\tilde\cc)\right\}$. We set $\A_\ha^{\haa}:=\A\!\rtimes_\ha^{\haa}\!\hG$\,, with generic elements ${\sf f,g}$ and algebraic laws $(\tilde\diamond,^{\tilde\diamond})$\,. The new coupling function is 
\begin{equation*}\label{jumol}
\tilde{\sf k}:\G\times\hG\to\U\M\big(\A_\ha^{\haa}\big)\,,\quad\ \tilde{\sf k}(x,\xi):=\delta_\varepsilon\otimes\k(x,\xi)^*\,.
\end{equation*}
The two twisted actions are defined similarly as above, by changing suitably the roles of the groups $\G$ and $\hG$\,. Explicitly one has
(here $1$ is the unit of $\M(\A)$ and ${\sf f}\in L^1(\hG;\A)$)\,:
\begin{equation}\label{pippin1}
[\c_x({\sf f})](\zeta)=\a_x[{\sf f}(\zeta)]\,\k(x,\zeta)\,,
\end{equation}
\begin{equation*}\label{baggins2}
\tilde\c_\xi({\sf f})=
(\delta_\xi\otimes 1)\,\tilde\diamond\,{\sf f}\,\tilde\diamond\,(\delta_\xi\otimes 1)^{\tilde\diamond}\,,
\end{equation*}
\begin{equation}\label{merry1}
\cc(x,y)=\delta_{\varepsilon}\otimes\aa(x,y)\,,
\end{equation}
\begin{equation*}\label{taleth2}
\tilde\cc(\xi,\eta)=(\delta_\xi\otimes 1)\,\tilde\diamond\,(\delta_\eta\otimes 1)\,\tilde\diamond\,(\delta_{\xi\eta}\otimes 1)^{\,\tilde\diamond} =\delta_{\varepsilon}\otimes\haa(\xi,\eta)\,.
\end{equation*}

Similarly as above one proves

\medskip
\begin{Theorem}\label{ailunindale}
If $\,\left\{(\A,\k),(\a,\aa),(\ha,\haa)\right\}$ is a given measurable (resp. continuous) covariant structure, then
$\big\{\big(\A\!\rtimes_\ha^\haa\tilde\G,\tilde{\sf k}\big),(\c,\gamma),(\tilde{\c},\tilde{\gamma})\big\}$ is a measurable (resp. continuous) $\hG$-particular covariant structure.
\end{Theorem}

\medskip
All the $2$-cocycles of the first generation are just tensor amplifications of those of the zero generation. At the level of actions, this is no longer true. But it does hold on certain $^*$-subalgebras, as shown by the next result.

\medskip
\begin{Lemma}\label{peregrin}
For every $x\in\G\,,\,\xi\in\hG$ and $m\in\M(\A)$ we have
\begin{equation}\label{anaire}
\b_x(\delta_{\sf e}\otimes m)=\delta_{\sf e}\otimes\a_x(m)\,,
\end{equation}
\begin{equation}\label{nienor}
\hb_\xi(\delta_{\sf e}\otimes m)=\delta_{\sf e}\otimes\ha_\xi(m)\,,
\end{equation}
\begin{equation}\label{elwing}
\c_x(\delta_{\varepsilon}\otimes m)=\delta_{\varepsilon}\otimes\a_x(m)\,,
\end{equation}
\begin{equation}\label{dior}
\tilde\c_\xi(\delta_{\varepsilon}\otimes m)=\delta_{\varepsilon}\otimes\ha_\xi(m)\,.
\end{equation}
\end{Lemma}

\begin{proof}
One has by (\ref{mistymountains}) and (\ref{huan})
$$
\begin{aligned}
\b_x(\delta_{\sf e}\otimes m)&=(\delta_x\otimes 1)\diamond(\delta_{\sf e}\otimes m)\diamond[\delta_{x^{-1}}\otimes \aa(x^{-1},x)^*]\\
&=[\delta_x\otimes \a_x(m)]\diamond[\delta_{x^{-1}}\otimes \aa(x^{-1},x)^*]\\
&=\delta_{\sf e}\otimes\left\{\a_x(m)\a_x\!\[\aa(x^{-1},x)^*\]\aa(x,x^{-1})\right\}\\
&=\delta_{\sf e}\otimes\a_x(m)\,,
\end{aligned}
$$
where the $2$-cocycle property of $\aa$ has been used for the last equality. 
To prove (\ref{nienor}) one must show for $g\in L^1(\G;\A)$ 
\begin{equation*}\label{zana}
\hb_\xi\!\[(\delta_{\sf e}\otimes m)\diamond g\]=\[\delta_{\sf e}\otimes \ha_\xi(m)\]\diamond \hb_\xi(g)\quad{\rm and}\quad\hb_\xi\!\[g\diamond(\delta_{\sf e}\otimes m)\]=\hb_\xi(g)\diamond\[\delta_{\sf e}\otimes \ha_\xi(m)\]\,.
\end{equation*}
This follows straightforwardly from (\ref{ossiriand}), (\ref{dorthonion}) and the definition of $\hb_\xi$\,.
Proving (\ref{elwing}) and (\ref{dior}) is similar.
\qed
\end{proof}

Starting from the covariant structure $\left\{(\A,\k),(\a,\aa),(\ha,\haa)\right\}$ and applying the twisted crossed product construction, we obtained new (particular) measurable covariant structures $\left\{\big(\A_\a^\aa,{\sf k}\big),(\b,\bb),(\hb,\hbb)\right\}$ and $\left\{\big(\A_\ha^{\haa},\tilde{\sf k}\big),(\mathrm c,\gamma),(\tilde\c,\tilde\cc)\right\}$\,. 
With all these objects one can construct (at least) two "second generation" $C^*$-algebras (they will be compared in the next section). First, one has 
\begin{equation*}\label{valinor}
\A_{\a,\hb}^{\aa,\hbb}\equiv\big(\A_{\a}^{\aa}\big)_\hb^\hbb:=(\A\rtimes_\a^\aa\G)\!\rtimes_\hb^{\hbb}\,\hG\,,
\end{equation*} 
with elements $F,G$ and algebraic structure $\big(\Box,^\Box\big)$\,. The second one is 
\begin{equation*}\label{nargothrond}
\A_{\ha,\c}^{\haa,\cc}\equiv\big(\A_{\ha}^{\haa}\big)_\c^\cc:=(\A\rtimes_\ha^{\haa}\hG)\!\rtimes_\c^{\cc}\G\,,
\end{equation*} 
with composition laws $\big(\tilde\Box,^{\tilde\Box})$ and elements ${\sf F,G}$\,. We recall that they also depend on  the coupling function $\k$\,. 

\medskip
\begin{Remark}\label{alatar}
{\rm There are other two (less interesting) second generation $C^*$-algebras
\begin{equation*}\label{vallinor}
\A_{\a,\b}^{\aa,\bb}\equiv(\A_{\a}^{\aa})_\b^\bb:=(\A\rtimes_\a^\aa\G)\!\rtimes_\b^{\bb}\G\quad{\rm and}\quad\A_{\ha,\tilde\c}^{\haa,\tilde\cc}\equiv(\A_{\ha}^{\haa})_{\tilde\c}^{\tilde\cc}:=(\A\rtimes_\ha^\haa\hG)\!\rtimes_{\tilde\c}^{\tilde\cc}\,\hG\,.
\end{equation*} 
}
\end{Remark}

\section{They are isomorphic}\label{cortandui}

The purpose now is to show that the second generation twisted crossed products $\A_{\a,\hb}^{\aa,\hbb}$ and $\A_{\ha,\c}^{\haa,\cc}$ are isomorphic and constitute realizations of the bi-product associated to a given covariant structure $\{(\A,\k),(\a,\aa),(\ha,\haa)\}$\,.

\medskip
\begin{Theorem}\label{eredluin}
There are one-to-one correspondences between:
\begin{enumerate}
\item
Covariant morphisms of the covariant structure $\left\{(\A,\k),(\a,\aa),(\ha,\haa)\right\}$\,.
\item
Non-degenerate morphisms of the $C^*$-algebra $\A_{\a,\hb}^{\aa,\hbb}$\,.
\item
Non-degenerate morphisms of the $C^*$-algebra $\A_{\ha,\c}^{\haa,\cc}$\,.
\end{enumerate}
\end{Theorem}

\begin{proof}
If $(\B,r,u,v)$ is given as in Definition \ref{stirba}, we are going to construct covariant morphisms 
$$
r_{u,v}:\A_{\a,\hb}^{\aa,\hbb}\to\M(\B)\quad{\rm and}\quad r_{v,u}:\A_{\ha,\c}^{\haa,\cc}\to\M(\B)\,.
$$
Using $(\B,r,u)$ we first construct the integrated form $r_u:=r\!\rtimes\!u:\A_{\a}^{\aa}\to\M(\B)$\,.
Let us check that $(\B,r_u,v)$ is a covariant morphism of $(\A_{\a}^{\aa},\hb,\hbb)$\,. First, for $f\in L^1(\G;\A)$  and $\xi\in\hG$ one has
$$
\begin{aligned}
v_\xi r_u(f)v_\xi^*&=\int_\G\!dx\,v_\xi r[f(x)]v_\xi^*\,v_\xi u_x v_\xi^*\\
&=\int_\G\!dx\,r\!\[\ha_\xi(f(x))\]r[\k(x,\xi)^*]\,u_x\\
&=\int_\G\!dx\,r\!\[(\hb_\xi f)(x)\]u_x=r_u\!\[\hb_\xi(f)\].
\end{aligned}
$$
Then, since $(\B,r,v)$ is a covariant representation of $(\A,\ha,\haa)$\,, for $\xi,\eta\in\hG$ we have $v_\xi v_\eta v_{\xi\eta}^*=r\!\[\haa(\xi,\eta)\]$.
Therefore it is enough to prove that $r_u\!\[\hbb(\xi,\eta)\]=r\!\[\haa(\xi,\eta)\]$\,. For $g\in L^1(\G;\A)$ one computes using (\ref{ossiriand})
$$
\begin{aligned}
r_u\!\[\hbb(\xi,\eta)\diamond g\]&=\int_\G\!dx\,r\big\{[(\delta_{\sf e}\otimes\haa(\xi,\eta))\diamond g](x)\big\}u_x\\
&=\int_\G\!dx\,r\!\left\{\haa(\xi,\eta)g(x)\right\}u_x=r\!\[\haa(\xi,\eta)\]r_u(g)\,.
\end{aligned}
$$
Similarly one gets $r_u\!\[g\diamond\hbb(\xi,\eta)\]=r_u(g)\,r\!\[\haa(\xi,\eta)\]$ and this is exactly what we needed to show.
Thus the (double) integrated form $r_{u,v}:=r_u\!\rtimes\!v=(r\!\rtimes\!u)\!\rtimes\!v$ is a non-degenerate morphism of $\A_{\a,\hb}^{\aa,\hbb}$\,. Analogously, $r_{v,u}:=r_v\!\rtimes\!u=(r\!\rtimes\!v)\!\rtimes\!u$ will be a nondegenerate morphism of $\A_{\ha,\c}^{\haa,\cc}$\,.

Now we show that every non-degenerate morphism $\mathcal R$ of $\A_{\a,\hb}^{\aa,\hbb}$ in some $C^*$-algebra $\B$ has the form $\mathcal R=(r\!\rtimes\!u)\!\rtimes\!v$ with $(\B,r,u,v)$ as required\,. The reasoning for non-degenerate morphisms $\mathcal S$ of $\A_{\ha,\c}^{\haa,\cc}$ would be similar.

The general theory, applied to the $C^*$-dynamical system $(\A_{\a}^{\aa},\hb,\hbb)$\,, tells us that $\mathcal R=R\!\rtimes\!v$ for some covariant morphism $(\B,R,v)$\,. In its turn, $R$ must have the form $r\!\rtimes\!u$ for a covariant morphism $(\B,r,u)$ of $(\A,\a,\aa)$\,. 
Let us show that $(\B,r,v)$ is a covariant morphism of $(\A,\ha,\haa)$\,. We already know that $v_\xi v_\eta=R\!\[\hbb(\xi,\eta)\]\!v_{\xi\eta}$\,.
So, to prove that $v_\xi v_\eta=r\!\[\haa(\xi,\eta)\]v_{\xi\eta}$ one needs to check that $R\!\[\hbb(\xi,\eta)\]=r\!\[\haa(\xi,\eta)\]$\,. But this has been done above.

On the other hand, by Lemma \ref{peregrin}, one has $\hb_\xi(\delta_{\sf e}\otimes A)=\delta_{\sf e}\otimes\ha_\xi (A)$ for every $\xi\in\hG$ and $A\in\A$\,. Thus one has
$$
v_\xi r(A)v_\xi^*=v_\xi R(\delta_{\sf e}\otimes A)v_\xi^*=R\!\[\hb_\xi(\delta_{\sf e}\otimes A)\]=R\!\[\delta_{\sf e}\otimes\ha_\xi (A)\]=r[\ha_\xi (A)]\,.
$$

Finally we show the right commutation relations between the unitary multipliers $u_x$ and $v_\xi$\,. The game is to deduce this only from the fact that $(\B,R,v)$ and $(\B,r,u)$ are covariant morphisms. 

Note first that elements of the form $\varphi\otimes\psi\otimes A$\,, with $A\in\A$\,, $\varphi\in L^1(\G)$ and $\psi\in L^1(\hG)$ (thus belonging to the algebraic tensor product $L^1(\G)\odot L^1(\hG)\odot\A$\,) are total in $\A_{\a,\hb}^{\aa,\hbb}$\,. Since $\mathcal R=R\!\rtimes\!v=(r\!\rtimes\!u)\!\rtimes\!v$\,, it is easy to check that $\mathcal R(\varphi\otimes\psi\otimes A)=r(A) u[\varphi]v[\psi]$\,, where we used the notations
$u[\varphi]:=\int_\G\!dx\,\varphi(x)u_x$ and $v[\psi]:=\int_{\hG}\!d\xi\,\psi(\xi)v_\xi$\,.
Thus, $\mathcal R$ being nondegenerate, it is enough to show for all the ingredients the identity 
\begin{equation*}\label{ansmak}
v_\xi u_x r(A)u[\varphi] v[\psi]=r[\k(x,\xi)^*]u_x v_\xi\,r(A)u[\varphi]v[\psi]\,.
\end{equation*}
Below, we are going to use the notation $g_x(\cdot):=\varphi(x^{-1}\cdot)\a_x(A)\aa(x,x^{-1}\cdot)\in L^1(\G;\A)$\,.
Using properties of the two covariant representations and axioms of the covariant structure, and recalling that $R=r\!\rtimes\!u$\,, we compute
$$
\begin{aligned}
v_\xi u_x\,r(A)u[\varphi] v[\psi]&=v_\xi r[\a_x(A)]u_x\!\int_\G dz\,\varphi(z)u_z\,v[\psi]\\
&=v_\xi r[\a_x(A)]\int_{\G}\!dy\,\varphi(x^{-1}y)r[\aa(x,x^{-1}y)]u_y v[\psi] \\
&=v_\xi \int_{\G}\!dy\, r\!\left\{\varphi(x^{-1}y)\a_x(A)\aa(x,x^{-1}y)\right\}\!u_y v[\psi] \\
&=v_\xi R(g_x) v[\psi]=R\[\hb_\xi(g_x)\] v_\xi v[\psi] \\
&=\int_{\G}\!dy\,r\!\left\{ \varphi(x^{-1}y)\,\ha_\xi\!\[\a_x(A)\aa(x,x^{-1}y)\]\k(y,\xi)^*\right\}u_y v_\xi v[\psi] \\
&=\int_{\G}\!dy \,\varphi(x^{-1}y)\,r\!\left\{\ha_\xi\!\[\a_x(A)\]\ha_\xi\!\[\aa(x,x^{-1}y)\]\k(y,\xi)^*\right\}u_y v_\xi v[\psi] \\
&\!\overset{(\ref{spoarca})}{=}\!\int_{\G}\!dy \,\varphi(x^{-1}y)\,r\!\left\{\ha_\xi\!\[\a_x(A)\]\k(x,\xi)^*\a_x[\k(x^{-1}y,\xi)^*]\aa(x,x^{-1}y)\right\}u_y v_\xi v[\psi] \\
&\!\overset{(\ref{dihanie})}{=}\!r[\k(x,\xi)^*]\,r\!\left\{\a_x\!\[\ha_\xi(A)\]\right\}\int_{\G}\!dz\,\varphi(z)\,r\big\{\a_x[\k(z,\xi)^*]\aa(x,z)\big\}u_{xz} v_\xi v[\psi]\\
&=r[\k(x,\xi)^*]\,r\!\left\{\a_x\!\[\ha_\xi(A)\]\right\}\int_{\G}\!dz\,\varphi(z)\,r\big\{\a_x[\k(z,\xi)^*]\big\}u_x\,u_z v_\xi v[\psi]\\
&=r[\k(x,\xi)^*]\,r\!\left\{\a_x\!\[\ha_\xi(A)\]\right\}u_x\!\int_{\G}\!dz\,\varphi(z)\,r[\k(z,\xi)^*]u_{z} v_\xi v[\psi]\\
&=r[\k(x,\xi)^*]\,u_x\,r\!\[\ha_\xi(A)\]\int_{\G}\!dz\,\varphi(z)\,r[\k(z,\xi)^*]u_{z} v_\xi v[\psi]\\
&=r[\k(x,\xi)^*]\,u_x\int_{\G}\!dz\,r\big\{\varphi(z)\ha_\xi(A)\k(z,\xi)^*\big\}u_{z} v_\xi v[\psi]\\
&=r[\k(x,\xi)^*]\,u_x\int_{\G}\!dz\,r\big\{\big[\hb_\xi(\varphi\otimes A)\big](z)\big\}u_{z} v_\xi v[\psi]\\
&=r[\k(x,\xi)^*]\,u_x R\big[\hb_\xi(\varphi\otimes A)\big] v_\xi v[\psi]\\
&=r[\k(x,\xi)^*]\,u_x v_\xi R(\varphi\otimes A) v[\psi]\\
&=r[\k(x,\xi)^*]\, u_x v_\xi r(A) u[\varphi]v[\psi]\,,
\end{aligned}
$$
so we are done.
\qed
\end{proof}

Then follows straightforwardly

\begin{Corollary}\label{motoc}
Both $\A_{\a,\hb}^{\aa,\hbb}$ and $\A_{\ha,\c}^{\haa,\cc}$ are bi-products of the covariant structure $\left\{(\A,\k),(\a,\aa),(\ha,\haa)\right\}$\,.
In particular, one has isomorphic $C^*$-algebras
\begin{equation*}\label{zmotoc}
\A^{\laa}_{\la}\cong\A_{\ra}^{\raa}\cong\A_{\a,\hb}^{\aa,\hbb}\cong\A_{\ha,\c}^{\haa,\cc}\,.
\end{equation*}
\end{Corollary}

Even if Corollary \ref{motoc} can be proved directly, it is interesting and useful to have explicit forms of the isomorphisms. Actually one has a commuting diagram of isomorphisms
$$
\begin{diagram}
\node{\A_{\a,\hb}^{\aa,\hbb}} \arrow{e,t}{\Upsilon} \arrow{s,l}{\Phi} \node{\A_{\ha,\c}^{\haa,\cc}} \arrow{s,r}{\Psi}  \\ 
\node{\A_{\ra}^{\raa}} \arrow{e,b}{\Gamma} \node{\A_{\la}^{\laa}}
\end{diagram}
$$
We have already specified $\Gamma$ before
\begin{equation*}\label{witch0}
\[\Gamma\big(\overleftarrow F\big)\]\!(x,\xi):=\overleftarrow F(x,\xi)\,\kappa(x,\xi)\,,
\end{equation*}
as a consequence of exterior equivalence of the twisted actions $(\la,\laa)$ and $(\ra,\raa)$\,. The actions of the other three on the $L^1$-Banach algebras are simply
\begin{equation*}\label{witch1}
[\Upsilon(F)(x)](\xi):=[F(\xi)](x)\,\k(x,\xi)\,,
\end{equation*}
\begin{equation*}\label{witch1}
[\Phi(F)](x,\xi):=[F(\xi)](x)\,,
\end{equation*}
\begin{equation*}\label{witch1}
[\Psi({\sf F})](x,\xi):=[{\sf F}(x)](\xi)\,,
\end{equation*}
and the diagram is already seen to commute. To convince the reader, we are going to exhibit the multiplications and the involutions of the iterated crossed products, at the level of $L^1$-elements. In $\A_{\a,\hb}^{\aa,\hbb}$ one has
$$
\begin{aligned}
&[(F\square G)(\xi)](x)
=\left\{\int_{\hG}\!d\eta\,F(\eta)\diamond\hb_\eta\!\[G(\eta^{-1}\xi)\]\diamond\hbb(\eta,\eta^{-1}\xi)\right\}\!(x)\\
=&\int_{\hG}\!d\eta\left\{F(\eta)\diamond\hb_\eta\!\[G(\eta^{-1}\xi)\]\diamond[\delta_e\otimes\haa(\eta,\eta^{-1}\xi)]\right\}\!(x)\\
=&\int_{\hG}\!d\eta\int_\G\!dy\,[F(\eta)](y)\,\a_y\!\[\!\(\hb_\eta[G(\eta^{-1}\xi)]\diamond[\delta_e\otimes\haa(\eta,\eta^{-1}\xi)]\)\!(y^{-1}x)\]\aa(y,y^{-1}x)\\
\overset{(\ref{dorthonion})}{=}\!\!&\,\int_{\hG}\!d\eta\int_\G\!dy\,[F(\eta)](y)\,\a_y\!\(\hb_\eta[G(\eta^{-1}\xi)](y^{-1}x)\,\a_{y^{-1}x}[\haa(\eta,\eta^{-1}\xi)]\)\aa(y,y^{-1}x)\\
=&\int_{\hG}\!d\eta\int_\G\!dy\,[F(\eta)](y)\,\a_y\big(\ha_\eta[G(\eta^{-1}\xi)(y^{-1}x)]\k(y^{-1}x,\eta)^*]\,\a_{y^{-1}x}[\haa(\eta,\eta^{-1}\xi)]\big)\aa(y,y^{-1}x)\\
=&\int_{\hG}\!d\eta\int_\G\!dy\,[F(\eta)](y)\,(\a_y\circ\ha_\eta)[G(\eta^{-1}\xi)(y^{-1}x)]\,\a_y[\k(y^{-1}x,\eta)^*]\,(\a_y\circ\a_{y^{-1}x})[\haa(\eta,\eta^{-1}\xi)]\big)\aa(y,y^{-1}x)\\
=&\int_{\hG}\!d\eta\int_\G\!dy\,[F(\eta)](y)\,(\a_y\circ\ha_\eta)[G(\eta^{-1}\xi)(y^{-1}x)]\,\a_y[\k(y^{-1}x,\eta)^*]\,\aa(y,y^{-1}x)\,\a_{x}[\haa(\eta,\eta^{-1}\xi)]\,,
\end{aligned}
$$
which should be compared with (\ref{guleras}) and
$$
\begin{aligned}
\[F^{\square}(\xi)\]\!(x)&=\left\{\Delta_{\hG}(\xi^{-1})\,\hbb(\xi,\xi^{-1})^{\diamond}\diamond\hb_\xi\!\[F(\xi^{-1})^{\diamond}\]\right\}\!(x)\\
&=\Delta_{\hG}(\xi^{-1})\left\{[\delta_{\sf e}\otimes\haa(\xi,\xi^{-1})^*]\diamond\hb_\xi\!\[F(\xi^{-1})^{\diamond}\]\right\}\!(x)\\
&=\Delta_{\hG}(\xi^{-1})\,\haa(\xi,\xi^{-1})^*\,\hb_\xi\!\[F(\xi^{-1})^{\diamond}\]\!(x)\\
&=\Delta_{\hG}(\xi^{-1})\,\haa(\xi,\xi^{-1})^*\,\ha_\xi\!\[F(\xi^{-1})^{\diamond}(x)\]\!\k(x,\xi)^*\\
&=\Delta_{\hG}(\xi^{-1})\,\haa(\xi,\xi^{-1})^*\,\ha_\xi\!\left\{\Delta_\G(x^{-1})\aa(x,x^{-1})^*\a_x\!\[F(\xi^{-1})(x^{-1})\]^*\right\}\!\k(x,\xi)^*\\
&=\Delta_{\hG}(\xi^{-1})\,\Delta_\G(x^{-1})\,\haa(\xi,\xi^{-1})^*\,\ha_\xi\!\left[\aa(x,x^{-1})^*\right](\ha_\xi\circ\a_x)\!\[F(\xi^{-1})(x^{-1})^*\]\!\k(x,\xi)^*\\
&\overset{(\ref{dihanie})}{=}\!\Delta_{\hG}(\xi^{-1})\,\Delta_\G(x^{-1})\,\haa(\xi,\xi^{-1})^*\,\ha_\xi\!\left[\aa(x,x^{-1})^*\right]\k(x,\xi)^*\,(\a_x\circ\ha_\xi)\!\[F(\xi^{-1})(x^{-1})^*\]\\
&\overset{(\ref{spoarca})}{=}\!\Delta_{\hG}(\xi^{-1})\,\Delta_\G(x^{-1})\,\haa(\xi,\xi^{-1})^*\,\aa(x,x^{-1})^*\,\a_x[\k(x^{-1},\xi)]\,(\a_x\circ\ha_\xi)\!\[F(\xi^{-1})(x^{-1})^*\]
\end{aligned}
$$
which should be compared with (\ref{gulerash}). In $\A_{\ha,\cc}^{\haa,\tilde\cc}$ one has
$$
\begin{aligned}
\,[({\sf F}\tilde\square{\sf G})(x)]&(\xi)
=\left\{\int_{\G}\!dy\,{\sf F}(y)\,\tilde\diamond\,\c_y\!\[{\sf G}(y^{-1}x)\]\tilde\diamond\,\cc(y,y^{-1}x)\right\}\!(\xi)\\
=&\int_{\G}\!dy\left\{{\sf F}(y)\,\tilde\diamond\,\c_y\!\[{\sf G}(y^{-1}x)\]\tilde\diamond\,[\delta_\varepsilon\otimes\aa(y,y^{-1}x)\right\}\!(\xi)\\
=&\int_{\G}\!dy\int_{\hG}\!d\eta\,[{\sf F}(y)](\eta)\,\ha_\eta\!\[\!\(\c_y[{\sf G}(y^{-1}x)]\,\tilde\diamond\,[\delta_\varepsilon\otimes\aa(y,y^{-1}x)]\)\!(\eta^{-1}\xi)\]\haa(\eta,\eta^{-1}\xi)\\
\overset{(\ref{dorthonion})}{=}\!\!&\,\int_{\G}\!dy\int_{\hG}\!d\eta\,[{\sf F}(y)](\eta)\,\ha_\eta\!\(\c_y[{\sf G}(y^{-1}x)](\eta^{-1}\xi)\,\ha_{\eta^{-1}\xi}[\aa(y,y^{-1}x)]\)\haa(\eta,\eta^{-1}\xi)\\
=&\int_{\G}\!dy\int_{\hG}\!d\eta\,[{\sf F}(y)](\eta)\,\ha_\eta\big(\a_y[{\sf G}(y^{-1}x)(\eta^{-1}\xi)]\k(y,\eta^{-1}\xi)]\,\ha_{\eta^{-1}\xi}[\aa(y,y^{-1}x)]\big)\haa(\eta,\eta^{-1}\xi)\\
=&\int_{\G}\!dy\int_{\hG}\!d\eta\,[{\sf F}(y)](\eta)\,(\ha_\eta\circ\a_y)[\G(y^{-1}x)(\eta^{-1}\xi)]\,\ha_\eta[\k(y,\eta^{-1}\xi)]\,(\ha_\eta\circ\ha_{\eta^{-1}\xi})[\aa(y,y^{-1}x)]\big)\haa(\eta,\eta^{-1}\xi)\\
=&\int_{\G}\!dy\int_{\hG}\!d\eta\,[{\sf F}(y)](\eta)\,(\ha_\eta\circ\a_y)[{\sf G}(y^{-1}x)(\eta^{-1}\xi)]\,\ha_\eta[\k(y,\eta^{-1}\xi)]\,\haa(\eta,\eta^{-1}\xi)\,\ha_{\xi}[\aa(y,y^{-1}x)]
\end{aligned}
$$
which should be compared with (\ref{gulas}), and
$$
\begin{aligned}
\[{\sf F}^{\tilde\square}(x)\]\!(\xi)&=\left\{\Delta_\G(x^{-1})\,\cc(x,x^{-1})^{\tilde\diamond}\,\tilde\diamond\,\c_x\!\[{\sf F}(x^{-1})^{\tilde\diamond}\]\right\}\!(\xi)\\
&=\Delta_\G(x^{-1})\left\{[\delta_\varepsilon\otimes\aa(x,x^{-1})^*]\,\tilde\diamond\,\c_x\!\[{\sf F}(x^{-1})^{\tilde\diamond}\]\right\}\!(\xi)\\
&=\Delta_\G(x^{-1})\,\aa(x,x^{-1})^*\,\c_x\!\[{\sf F}(x^{-1})^{\tilde\diamond}\]\!(\xi)\\
&=\Delta_\G(x^{-1})\,\aa(x,x^{-1})^*\,\a_x\!\[{\sf F}(x^{-1})^{\tilde\diamond}(\xi)\]\kappa(x,\xi)\\
&=\Delta_\G(x^{-1})\,\aa(x,x^{-1})^*\,\a_x\!\left\{\Delta_{\hG}(\xi^{-1})\haa(\xi,\xi^{-1})^*\ha_\xi[{\sf F}(x^{-1})(\xi^{-1})^*]\right\}\kappa(x,\xi)\\
&=\Delta_\G(x^{-1})\,\Delta_{\hG}(\xi^{-1})\,\aa(x,x^{-1})^*\,\a_x\!\left[\haa(\xi,\xi^{-1})^*\](\a_x\circ\ha_\xi)\!\[{\sf F}(x^{-1})(\xi^{-1})^*\]\kappa(x,\xi)\\
&\overset{(\ref{dihanie})}{=}\!\Delta_\G(x^{-1})\,\Delta_{\hG}(\xi^{-1})\,\aa(x,x^{-1})^*\a_x\!\left[\haa(\xi,\xi^{-1})^*\]\kappa(x,\xi)\,(\ha_\xi\circ\a_x)[{\sf F}(x^{-1})(\xi^{-1})^*]\\
&\overset{(\ref{poarca})}{=}\!\Delta_\G(x^{-1})\,\Delta_{\hG}(\xi^{-1})\,\aa(x,x^{-1})^*\haa(\xi,\xi^{-1})^*\,\ha_\xi\!\[\k(x,\xi^{-1})^*\](\ha_\xi\circ\a_x)\[{\sf F}(x^{-1})(\xi^{-1})^*\]\\
\end{aligned}
$$
which should be compared with (\ref{gulash}).

\medskip
\begin{Remark}\label{boilup}
{\rm If one tries to show directly that $\Upsilon$ is multiplicative, after a short computation using (\ref{dihanie}), he will realize that this is equivalent to the identity (\ref{vrajitroaca}). 
}
\end{Remark}

\medskip
\begin{Remark}\label{cazan}
{\rm Naturally, by the same mechanism, the second generation $C^*$-algebras can also be inflated to new covariant structures $\left\{\big(\A_{\a,\hb}^{\aa,\hbb},{\sf k}^2\big),(\b^{2},\bb^{2}),(\hb^{2},\hbb^{2})\right\}$ and $\left\{\big(\A_{\ha,\c}^{\haa,\cc},\tilde{\sf k}^2\big),(\c^{2},\cc^{2}),(\tilde\c^{2},\tilde\cc^{2})\right\}$\,. Then the isomorphism $\Upsilon$ can be upgraded to an isomorphism in a category of covariant structures, that can be easily defined. Similarly, the twisted crossed products $\A_{\la}^{\laa}$ and $\A_{\ra}^{\raa}$ with product group $\G\times\hG$ also have their natural covariant structures and the isomorphisms $\Gamma,\Phi$ and $\Psi$ have their interpretation in this category. Since many formulas should be written down and also having in view a subsequent work, we shall not pursue all these here.
}
\end{Remark}

\section{Takai duality and other examples}\label{corsandui}

\begin{Example}\label{lapusa}
{\rm We have seen that one realization of the bi-product $\A_{(\a,\ha)}^{(\aa,\haa)}$ is the twisted crossed product $\A_{\la}^{\laa}:=\A\!\rtimes_{\la}^{\laa}(\G\times\hG)$\,. Applying to this one known results \cite{Sm1}, it follows that the bi-product is commutative if and only if $\A,\G,\hG$ are commutative, $\a$ and $\ha$ are trivial and $\la$ is (essentially) symmetric. But $\la$ is symmetric if and only if $\aa$ and $\haa$ are symmetric and $\k=1$\,.
}
\end{Example}

\medskip
\begin{Example}\label{somnostrika}
{\rm If $\k=1$ the two actions $\a$ and $\ha$ commute, the elements $\haa(\xi,\eta)$ are fixed points of $\a$\,, the elements $\aa(x,y)$ are fixed points of $\ha$\,, one has $\aa(x,y)\haa(\xi,\eta)=\haa(\xi,\eta)\aa(x,y)$ and the twisted actions $(\la,\laa)$ and $(\ra,\raa)$ coincide.
The isomorphism between $\A_{\a,\hb}^{\aa,\hbb}$ and $\A_{\ha,\c}^{\haa,\cc}$ is basically a flip of the variables. The twisted actions $(\hb,\hbb)$ and $(\c,\cc)$ are non-trivial only in the $\A$-part of the twisted crossed products. 
}
\end{Example}

\medskip
\begin{Example}\label{somnorika}
{\rm If the initial two actions are not twisted, i.e. $\aa=1$ and $\haa=1$\,, then $\k$ must verify for all $x,y,\xi,\eta$
\begin{equation}\label{visila}
\k(x,\xi\eta)=\k(x,\xi)\ha_\xi[\k(x,\eta)]\quad{\rm and}\quad\k(xy,\xi)^*=\k(x,\xi)^*\a_x[\k(y,\xi)^*]\,.
\end{equation}
This means that $\kappa(x,\cdot):\hG\to\U\M(\A)$ and $\kappa(\cdot,\xi)^*:\G\to\U\M(\A)$ are crossed morphisms. One has
\begin{equation}\label{cracacila}
\laa\big((x,\xi),(y,\eta)\big)=\ha_\xi[\k(x,\eta)]\,,\quad\raa\big((x,\xi),(y,\eta)\big)=\a_x[\k(y,\xi)^*]\,.
\end{equation} 
The $\A_{\la}^{\laa}$-realization of the bi-product $\A_{(\a,\ha)}^{(\aa,\haa)}$ is still twisted and can be very complicated. The iterated crossed products $\A_{\a,\hb}^{\aa,\hbb}\equiv\A_{\a,\hb}$ and $\A_{\ha,\c}^{\haa,\cc}\equiv\A_{\ha,\c}$ are only constructed with untwisted actions, but the actions $\hb,\c$\,, besides the initial $\ha,\a$\, also contain the coupling function $\k$\,.
}
\end{Example}

\medskip
\begin{Example}\label{somnorila}
{\rm Even when both twisted actions are trivial, the bi-product remembers the $C^*$-algebra $\A$ and the "coupling" between the groups $\G$ and $\hG$\,. For $\left\{(\A,\k),({\sf id},1),({\sf id},1)\right\}$ one gets $\la={\sf id}$ but
\begin{equation}\label{cacacila}
\laa\big((x,\xi),(y,\eta)\big)=\k(x,\eta)
\end{equation} 
is still non-trivial. Relations (\ref{poarca}) and (\ref{spoarca}) become in this case (respectively)
$$
\k(x,\xi\eta)=\k(x,\xi)\k(x,\eta)\quad{\rm and}\quad\k(xy,\xi)=\k(y,\xi)\k(x,\xi)\,.
$$ 
For Abelian $\A$, twisted crossed products $\A\!\rtimes_{\sf id}^{\laa}\!{\sf H}$ with trivial action $\la$ (but with general $2$-cocycle $\laa$) have been studied in depth in \cite{Sm1,Sm2,EW}. It is worth mentioning that our $\laa$ is symmetric only if $\k=1$\,. The second generation iterated twisted crossed products have the form 
$$
(\A\!\rtimes_{\sf id}\G)\!\rtimes_{\hb^\bullet}\!\hG\cong[\A\otimes C^*(\G)]\!\rtimes_{\hb^\bullet}\!\hG\quad{\rm and}\quad
(\A\!\rtimes_{\sf id}\hG)\!\rtimes_{\c^\bullet}\!\G\cong[\A\otimes C^*(\hG)]\!\rtimes_{\c^\bullet}\!\G\,,
$$
where essentially $[\hb^\bullet_\xi(f)](x):=f(x)\kappa(x,\xi)^*$ and $[c^\bullet_x({\sf f})](\xi):={\sf f}(\xi)\k(x,\xi)$\,.

If $\k$ is $\T$-valued, $\laa$ is a bi-character. It is easy to see that we get
\begin{equation}\label{cascacila}
\A_{({\sf id},1)}^{({\sf id},1)}\equiv\A_{\sf id}^{\laa}\cong\A\otimes C^*_{\k}(\G\times\hG)\,.
\end{equation}
We denoted by $C^*_{\k}(\G\times\hG)$ the twisted group algebra of ${\sf H}:=\G\times\hG$ corresponding to the $2$-cocycle ${\sf H}\times{\sf H}\to\T$ given by (\ref{cacacila}). 
More generally, we can consider the covariant structure $\left\{(\A,\k),({\sf id},\aa),({\sf id},\haa)\right\}$\,, where $\aa$ and $\haa$ are multipliers (they take values in $\T$)\,. If $\k$ is also $\T$-valued, then 
\begin{equation}
\A_{({\sf id},{\sf id})}^{(\aa,\haa)}\cong \A\otimes C^*_{\laa}(\G\times\hG)\,.
\end{equation}
}
\end{Example}

\medskip
\begin{Example}\label{pasarila}
{\rm We shall describe now briefly how {\it a twisted version of Takai's duality result for Abelian groups} follows from our isomorphism $\A_{\a,\hb}^{\aa,\hbb}\cong\A_{\ha,\c}^{\haa,\cc}$\,, which is written with full notations 
\begin{equation}\label{harrypotter}
\(\A\!\rtimes_\a^\aa\!\G\)\!\rtimes_\hb^\hbb\hG\cong\big(\A\!\rtimes_\ha^\haa\!\hG\big)\!\rtimes_\c^{\cc}\!\G\,.
\end{equation}
Let us suppose that the group $\G$ is commutatative (in additive notations) and $\hG:=\widehat\G$ is its Pontryagin dual. As coupling function we choose the natural duality $\k(x,\xi)\equiv\k^0(x,\xi):=\xi(x)$\,. Also assume that the initial twisted action of $\widehat\G$ is trivial: $(\ha,\haa)=({\sf id},1)$\,; then 
the $2$-cocycle $\hbb$ is trivial and the action $\hb$ reduces to the standard dual action given by $\big[\hat\b^0_\xi(f)\big](x):=\overline{\xi(x)}f(x)$\,.
The purpose is to express the double twisted crossed product $\(\A\!\rtimes_\a^\aa\!\G\)\!\rtimes_{\hat\b^0}\!\widehat\G$ in a simple familiar form, using the r.h.s. of (\ref{harrypotter}).

There are well-known canonical isomorphisms $\A\!\rtimes_{\sf id}^1\!\widehat\G\cong\A\otimes C^*\!\big(\widehat\G\big)\cong\A\otimes C_0(\G)$\,, the second one being given by a partial Fourier transform. The twisted action $(\c,\cc)$ given by (\ref{pippin1}) and (\ref{merry1}) is carried to $(\a\otimes{\sf t},\aa\otimes 1)$\,, where $[{\sf t}_x(\varphi)](y):=\varphi(y+x)$ is the action of $\G$ on $C_0(\G)$ by translations. If one finds an isomorphism
\begin{equation}\label{semifinala}
[\A\otimes C_0(\G)]\rtimes_{\a\otimes{\sf t}}^{\aa\otimes 1}\G\cong\A\otimes[C_0(\G)\!\rtimes_{\sf t}\!\G]\,,
\end{equation}
then using the standard isomorphism between $C_0(\G)\!\rtimes_{\sf t}\!\G$ and the $C^*$-algebra $\mathbb K[L^2(\G)]$ of all compact operators in the Hilbert space $L^2(\G)$ one finally gets the desired result
\begin{equation}\label{takai}
\(\A\!\rtimes_\a^\aa\!\G\)\!\rtimes_{\hat\b^0}\!\widehat\G\cong\A\otimes\mathbb K[L^2(\G)]\,.
\end{equation}
Using some notational abuse, the isomorphism (\ref{semifinala}) is given by
$$
[\Theta(F)](z,x):=\a_x[F(z,x)]\aa(x,z)\,.
$$
We refer to \cite[Sect. 7.1]{Wi} for a more careful discussion of the case $\aa=1$\,.

The conclusion is that in this case the bi-product associated to the covariant structure $\{(\A,\k^0),(\a,\aa),({\sf id},1)\}$ is stable equivalent to the initial $C^*$-algebra $\A$\,. Recalling the realizations $\A_{\la}^{\laa}$ and $\A_{\ra}^{\raa}$ of this bi-product, we get more isomorphisms that could be of some interest. In the present given situation, for example, one has
$$
\la_{\!(x,\xi)}=\a_x\,,\quad\laa\big((x,\xi),(y,\eta)\big)=\eta(x)\aa(x,y)\,.
$$
For this twisted action one gets $\A\!\rtimes_{\la}^{\laa}\!(\G\times\widehat\G)\cong\A\otimes\mathbb K[L^2(\G)]$\,.

All the isomorphisms we described above are shadows of isomorphisms of covariant systems, as indicated in Remark \ref{cazan}.
}
\end{Example}

\bigskip
{\bf Acknowledgements:}
The authors are supported by {\it N\'ucleo Milenio de F\'isica Matem\'atica RC120002}. M. M. acknowledges support from the Fondecyt Project 1120300.

\bigskip


\bigskip
\bigskip
\medskip
{\bf Address}

\medskip
Departamento de Matem\'aticas, Universidad de Chile,

Las Palmeras 3425, Casilla 653, Santiago, Chile

\emph{E-mail:} h.bustos1988@gmail.com

\emph{E-mail:} mantoiu@uchile.cl

\end{document}